\documentclass[11pt,a4paper]{article}
\usepackage{amsmath,amssymb,amsthm,amscd}
\usepackage{indentfirst}
\usepackage[top=25mm, bottom=30mm, left=25mm, right=25mm]{geometry}
\newtheorem{Def}{\bf Definition}[subsection]
\newtheorem{Thm}[Def]{\bf Theorem}
\newtheorem{Lem}[Def]{\bf Lemma}
\newtheorem{Cor}[Def]{\bf Corollary}
\newtheorem{Pro}[Def]{\bf Proposition}
\newtheorem{Rem}[Def]{\bf Remark}
\newtheorem{Exa}[Def]{\bf Example}

\newcommand{\hooklongrightarrow}{\lhook\joinrel\longrightarrow}
\begin{document}
\title{\bf Weak Exactness for $\bf C^*$-algebras and Application to Condition (AO)}
\author{Yusuke Isono\thanks{Department of Mathematical Sciences,
University of Tokyo, Komaba, Tokyo, 153-8914, Japan \protect \\ \hspace{7.5pt}\quad E-mail: \texttt{isono@ms.u-tokyo.ac.jp}}}
\date{}

\maketitle

\begin{abstract}
Kirchberg introduced weak exactness for von Neumann algebras as an analogue of exactness for $C^*$-algebras. Ozawa found useful characterizations of weak exactness and he proved that weak exactness for group von Neumann algebras is equivalent to exactness of the groups. 
We generalize this concept to inclusions of $C^*$-algebras in von Neumann algebras and study some fundamental properties and relationship to Kirchberg's weak exactness. We then give some permanence properties which are similar to those of exact groups. 
In the last section, we study a similar condition to Ozawa's condition (AO) with our weak exactness and generalize Ozawa's theorem for bi-exact groups. As a corollary, we give new examples of prime factors. 
\end{abstract}

\section{\bf Introduction}
	
Weak exactness for von Neumann algebras is a notion introduced by Kirchberg $\cite{wk ex K}$. Ozawa developed the theory and gave some fundamental characterizations $\cite{wk ex O}$ and Dong--Ruan generalized the notion for dual operator spaces $\cite{DR}$.

In particular, Ozawa proved that a group von Neumann algebra $L\Gamma$ is weakly exact if and only if the group $\Gamma$ is exact. 
It is natural to expect that weak exactness has permanence properties which correspond to those of groups. For example, exactness of groups is preserved under taking subgroups, extensions, free products, and increasing unions, and so we may expect so does weak exactness. 
Indeed, it is well-known that exactness of $C^{\ast}$-algebras is preserved under taking subalgebras, tensor products, crossed products with exact groups, free products, and increasing unions. However, the thing is not so simple in the von Neumann algebra context.

This difficulty comes from the fact that weak exactness does $\it not$ pass to subalgebras (subspaces). Indeed, $\mathbb{B}(H)$, the von Neumann algebra of all bounded linear operators on a Hilbert space $H$, is always weakly exact, but there exist non weakly exact von Neumann algebras since there exist non-exact groups $\cite{Gr}$, $\cite{ame O}$. 
For this reason, there is no hope to find any local property for weak exactness. We mention that exactness for $C^{\ast}$-algebras has useful local characterizations.

Here is another important reason for the difficulty. To explain this, let us recall weak nuclearity. Let $A$ be a $C^{\ast}$-algebra, $M$ be a von Neumann algebra, and $\phi\colon A\rightarrow M$ be a linear map. 
We say $\phi$ is $\it factorable$ if it factors through a matrix algebra with completely positive (we write as c.p.) maps; $\phi$ is $\it weakly$ $\it nuclear$ if it is approximated by a net of contractive factorable maps point $\sigma$-weakly; and $M$ is $\it semidiscrete$ if $\mathrm{id}_{M}$ is weakly nuclear. 
Kirchberg proved $\cite{Ki 1}$ that $\phi\colon A\rightarrow M\subset\mathbb{B}(H)$ is weakly nuclear if and only if the map $A\odot M'\ni a\otimes x\mapsto ax\in\mathbb{B}(H)$ is continuous with respect to the minimal tensor norm (written as ''min-continuous'' or ''min-bounded'').
It is proved that semidescreteness is preserved under increasing unions but the proof involves a deep result, namely, equivalence of semidescreteness and injectivity. For a von Neumann algebra $M$ and an increasing net $(M_{i})$ of semidiscrete subalgebras with $M=(\bigcup_{i}M_{i})''$, it is easy to show that the inclusion from the norm closure of $\bigcup_{i}M_{i}$ to $M$ is weakly nuclear, but it is not straightforward to extend this property on $M$. The same phenomenon happens when we prove it for weak exactness. 

To solve the problem on semidiscreteness, we need the following deep fact $\cite{CE}$:
\begin{itemize}
\item for a von Neumann algebra $M$ and a $\sigma$-weakly dense $C^{\ast}$-subalgebra $A$ of $M$, $M$ is semidiscrete if the inclusion $A\hookrightarrow M$ is weakly nuclear.
\end{itemize}
We note that this fact is closely related to the following important theorems: semidiscreteness is equivalent to injectivity $\cite{Co}$, $\cite{CE}$; nuclearity (or exactness) implies local reflexivity; and exactness is equivalent to property $C$ $\cite{ex C}$.

In this paper, we will go along a similar line to the above. We first generalize weak exactness to $C^{\ast}$-algebras with faithful representations and reformulate Ozawa's characterizations in this context. We then prove the following result similar to the above.
\begin{Thm}[Corollary $\ref{wk ex main}$]
For a von Neumann algebra $M$ with separable predual and a $\sigma$-weakly dense $C^{\ast}$-subalgebra $A$ of $M$, $M$ is weakly exact if $A$ is weakly exact in $M$ (see Definition $\ref{def wk ex}$).
\end{Thm}
 By this theorem, we obtain the following permanence properties.
\begin{Cor}[Propositions $\ref{inc uni}$, $\ref{ten pro}$, and $\ref{cro pro}$]
For von Neumann algebras with separable predual, weak exactness is preserved under tensor products, crossed products with exact groups, and increasing unions.
\end{Cor}
For free products, we have only partial answers and the general case is left open.

In the last section, we study a similar condition to Ozawa's condition (AO). More precisely, we will replace local reflexivity in condition (AO) with our weak exactness. As an application, we generalize Ozawa's theorem from von Neumann algebras of finite type to those of general type. 
We refer to Section $\ref{ap ao}$ for the details. In the theorem below, we need our weak exactness for the proof of the first statement. 

\begin{Thm}[Corollaries $\ref{ab prim}$ and $\ref{inj prim}$]
Let $A$ be an injective finite von Neumann algebra with separable predual and $\Gamma$ be a countable discrete group acting on $A$. Let $M$ be a non-injective (possibly non-unital) subfactor of $A\mathbin{\bar{\rtimes}}\Gamma$ with a faithful normal conditional expectation from $1_{M}(A\mathbin{\bar{\rtimes}}\Gamma)1_{M}$ onto $M$. Assume $\Gamma$ is bi-exact (see Definition $\ref{bi-ex def}$).
\begin{itemize}
\item[$(1)$]If $M$ is not of type ${\rm III}_1$, then it is semiprime. That means, for any tensor decomposition $M=M_{1}\mathbin{\bar{\otimes}}M_{2}$, one of $M_{i}$ ($i=1,2$) is injective.
\item[$(2)$]If $A$ is abelian, then $M$ is prime. That means, for any tensor decomposition $M=M_{1}\mathbin{\bar{\otimes}}M_{2}$, one of $M_{i}$ ($i=1,2$) is of type $\rm I$.
\end{itemize}
\end{Thm}
In Subsection $\ref{ex prim}$, we give some explicit examples of (semi)prime factors as a corollary of the theorem.

After writing a draft of this paper, the author was informed that C. Houdayer and S. Vaes also obtained the same result as the second statement in this theorem by a very different method $\cite{HV}$. They used $\it array$, a generalized notion of cocycle defined in $\cite{CS}$, and so their proof is completely different from ours. We mention that bi-exactness is equivalent to a condition on arrays $\cite[\rm Proposition\ 2.1]{CSB}$.
\\[1em]
A good reference throughout the paper is a book $\cite{BO}$ of Brown and Ozawa.
\\[1em]
${\bf \bullet Notation}$

\nopagebreak
We denote Hilbert spaces by $H,K$; the sets of bounded linear operators on them by $\mathbb{B}(H),\mathbb{B}(K)$ (or sometimes $\mathbb{B}$); and the Hilbert space tensor product of them by $H\otimes K$. The $(n\times n)$-matrix algebra with coefficient in $\mathbb{C}$ is written as $\mathbb{M}_{n}$. 
The minimal tensor product of $C^{\ast}$-algebras $A,B$ is denoted by $A\otimes B$ and the algebraic tensor product is denoted by $A\odot B$. For a discrete group $\Gamma$, $A\rtimes_{r} \Gamma$ means the reduced crossed product and $A\rtimes \Gamma$ means the full crossed product. 
The crossed product von Neumann algebra for a von Neumann algebra $M$ is denoted by $M\mathbin{\bar{\rtimes}}\Gamma$. The symbol $A\ast B$ is the reduced free product $C^{\ast}$-algebra and $M\mathbin{\bar{\ast}} N$ is the free product von Neumann algebra. Finally, we use the symbol $J\lhd B$ for a pair of a unital $C^{\ast}$-algebra $B$ and a closed ideal $J$ of $B$.

\section{\bf Preliminaries}\label{pre}

\subsection{\bf Operator spaces}\label{op sp}

For operator space theory, we refer the reader to $\cite{ER}$, $\cite{Pisier}$.

A Banach space $X$ is said to be an $\it operator$ $\it space$ if it is a closed subspace of a $C^{\ast}$-algebra $A$. For an operator space $X$, we endow a norm on $X\odot \mathbb{M}_{n}$ with $A\otimes \mathbb{M}_{n}$ for each $n\in \mathbb{N}$ and we consider the operator space structure with such norms which depends on the choice of $A$. 
The operator space $X\odot \mathbb{M}_{n}$ with the norm from $A\otimes \mathbb{M}_{n}$ is denoted by $X\otimes \mathbb{M}_{n}$ or $\mathbb{M}_{n}(X)$. More generally, the operator space $X\otimes Y$ is the compression of $X\odot Y$ in $A\otimes B$ for operator spaces $X\subset A$ and $Y\subset B$.

A bounded linear map $\phi \colon X\rightarrow Y$ between operator spaces is said to be $\it completely$ $\it bounded$ ($\it c.b.$) if $\sup_{n}\|\phi\otimes\mathrm{id}_{n}\|(=:\|\phi\|_{{\rm cb}})<\infty$ for bounded linear maps $\phi\otimes\mathrm{id}_{n}: X\otimes \mathbb{M}_{n} \rightarrow Y\otimes \mathbb{M}_{n}$. 
We say the map $\phi$ is $\it completely$ $\it contractive$ ($\it c.c.$) if $\|\phi\|_{{\rm cb}}\leq1$; $\phi$ is $\it completely$ $\it isometry$ if each $\phi\otimes\mathrm{id}_{n}$ is isometry; and $\phi$ is $\it completely$ $\it isomorphic$ if it is invertible and $\|\phi\|_{{\rm cb}}=\|\phi^{-1}\|_{{\rm cb}}=1$. We denote the space of all c.b. maps from $X$ into $Y$ by $\mathrm{CB}(X,Y)$ equipped with a  Banach norm $\|\cdot\|_{{\rm cb}}$. It is an operator space in a natural way. 

For an operator space $X$, the $\it dual$ $\it operator$ $\it space$ $X^{\ast}$ of $X$ is the dual space of $X$ as a Banach space. We give an operator space structure to $X^{\ast}$ by 
\begin{equation*}
X^{\ast} \ni \phi \longmapsto (\theta_{x}(\phi))_{x}\in \prod_{m\in\mathbb{N}} \prod_{x\in(\mathbb{M}_{m}(X))_{1}} \mathbb{M}_{m}\subset \mathbb{B},
\end{equation*}
where $\theta_{x}$ is the map defined by $\theta_{x} \colon X^{\ast}\ni \phi \mapsto (\phi(x_{k,l}))_{k,l}\in \mathbb{M}_{m}$ for $x=(x_{k,l})_{k,l}\in\mathbb{M}_{m}$. This structure gives the following isometric isomorphism:
\begin{equation*}
X^{\ast}\otimes\mathbb{M}_{n}\ni \phi\otimes e_{k,l}\longmapsto(x\mapsto \phi(x)e_{k,l})\in \mathrm{CB}(X,\mathbb{M}_{n}).
\end{equation*}
We can generalize this correspondence by replacing $\mathbb{M}_{n}$ with any operator space $Y$. Indeed, for $\phi\in X^{\ast}$ and $y\in Y$, there is a finite rank map  $T_{\phi,y}\colon X\rightarrow Y$ defined by $T_{\phi,y}(x)=\phi(x)y$, and this gives the following isometric inclusion:
\begin{equation*}
X^{\ast}\otimes Y \hooklongrightarrow \mathrm{CB}(X,Y).
\end{equation*}
The inclusion is surjective if $X$ or $Y$ is finite dimensional. In the case where $X$ is a  dual operator space $Z^{\ast}$, the image of $Z\odot Y\subset Z^{\ast\ast}\otimes Y$ in $\mathrm{CB}(Z^{\ast},Y)$ coincides with the set of all finite rank $\mathrm{weak}^{\ast}$-continuous maps.

Next, we define a direct limit of operator spaces. Let $(X_{n},\phi_{n})_{n\in\mathbb{N}}$ be a direct system of operator spaces, that is, each $X_{n}$ is an operator space and each $\phi_{n} \colon X_{n} \rightarrow X_{n+1}$ is a c.c. map. For simplicity, we assume $X_{n}\subset \mathbb{M}_{k(n)}$ for some $k(n)\in\mathbb{N}$ (in the paper, we will need only such $X_{n}$). Let $S_{0}$ be a space
\begin{equation*}
\{ (x_{n})_{n}\in \prod_{n}\mathbb{M}_{k(n)} \mid {\rm there\ exists}\ m\in\mathbb{N}\  {\rm such \ that}\  x_{m}\in X_{m} \ {\rm and}\  x_{n+1}=\phi_{n}(x_{n}) \ {\rm for\ all}\ n\geq m\}
\end{equation*}
and define the $\it direct$ $\it limit$ of $(X_{n},\phi_{n})$ as $\pi (S)$ and denote it by $\displaystyle\varinjlim_{n}(X_{n},\phi_{n})$, where $\pi \colon \prod \mathbb{M}_{k(n)}\rightarrow \prod\mathbb{M}_{k(n)}/\bigoplus\mathbb{M}_{k(n)}$ is the quotient map and $S$ is the norm closure of $S_{0}$. The space $\displaystyle\varinjlim_{n}(X_{n},\phi_{n})$ is automatically closed (use approximate unit of $\bigoplus\mathbb{M}_{k(n)}$) and so it is an operator space. We will see $\displaystyle\varinjlim_{n}(X_{n},\phi_{n})$ is 1-exact as an operator space.

\subsection{\bf Double dual of $C^{\ast}$-algebras}\label{**}

Let $M$ be a von Neumann algebra and $A\subset M$ be a $\sigma$-weakly dense unital $C^{\ast}$-subalgebra (possibly $A=M$). 
Then the identity map on $A$ can be uniquely extended to a map from $A^{\ast\ast}$ onto $M$ and we write the central projection corresponding the kernel of the map as $1-z_{M}$. That means $z_{M}$ is the unique central projection satisfying $A^{\ast\ast}z_{M}\simeq M$ (see $\cite[\rm Chapter\ III]{Tak 1}$, for example). 
We note that the multiplication map by $z_M$ from $A$ to $A^{**}z_M$ is injective and so, in the case where $A=M$, we have $Mz_{M}\simeq M$.

Let $M$ be a von Neumann algebra and consider $z_M$ for $A=M$. If a net $z_i$ converges to $z_M$ $\sigma$-weakly in $M^{**}$, this net converges to 1 $\sigma$-weakly in $M$ by the identification $Mz_M\simeq M$. For any $\phi\in\mathrm{Aut}(M)$ and its unique extension $\widetilde{\phi}\in\mathrm{Aut}(M^{**})$, we have $\displaystyle\lim_i \phi(z_i)=1$ in $M$ and $\displaystyle\lim_i \widetilde{\phi}(z_i)=\widetilde{\phi}(z_M)$ in $M^{**}$. Then, multiplying by $z_M$, we have $\displaystyle\lim_i \phi(z_i)z_M=z_M$ in $Mz_M\subset M^{**}$ and $\displaystyle\lim_i \widetilde{\phi}(z_i)z_M=\widetilde{\phi}(z_M)z_M$ in $M^{**}$. Since $\phi(z_i)=\widetilde{\phi}(z_i)$ in $M^{**}$, these equations implies $\widetilde{\phi}(z_M)z_M=z_M$ and so $\widetilde{\phi}(z_M)\geq z_M$. Considering $\phi^{-1}$, we proved that $\widetilde{\phi}(z_M)=z_M$ for any $\phi\in\mathrm{Aut}(M)$.

Let $0\rightarrow J\rightarrow A \rightarrow A/J\rightarrow 0$ be an exact sequence and $\pi \colon A\rightarrow A/J$ be the quotient map. Take the central projection $z\in A^{\ast\ast}$ with $zA^{\ast\ast}\simeq J^{\ast\ast}$. Then we have a natural identification
\begin{alignat*}{3}
(1-z)A^{\ast\ast} &\longrightarrow& \ (A/J)^{\ast\ast}  &\longrightarrow& \ A^{\ast\ast}/J^{\ast\ast} \\
(1-z)a \hspace{1.1em} &\longmapsto& \pi^{\ast\ast}(a) \  &\longmapsto& \ a+J^{\ast\ast},
\end{alignat*}
where $\pi^{\ast\ast} \colon A^{\ast\ast}\rightarrow (A/J)^{\ast\ast}$ is the extension of $\pi$.

Let $A$ and $B$ be unital $C^{\ast}$-algebras. For the inclusion $A\simeq A\otimes1\subset A\otimes B$, there is the canonical extension $A^{\ast\ast}\simeq A^{\ast\ast}\otimes1 \hookrightarrow (A\otimes B)^{\ast\ast}$.
 Note that the inclusion is homeomorphic in the $\sigma$-weak (and hence $\sigma$-strong and $\sigma$-$\mathrm{strong}^{\ast}$) topology. Since ranges of $A^{\ast\ast}\hookrightarrow (A\otimes B)^{\ast\ast}$ and $B^{\ast\ast}\hookrightarrow (A\otimes B)^{\ast\ast}$ commute with each other, they induce the following injective $\ast$-homomorphism:
\begin{equation*}
A^{\ast\ast}\odot B^{\ast\ast} \longrightarrow (A\otimes B)^{\ast\ast}.
\end{equation*}
Here the injectivity of this map comes from that of $A^{\ast\ast}\odot B^{\ast\ast} \rightarrow (A\otimes B)^{\ast\ast} \rightarrow A^{\ast\ast}\mathbin{\bar{\otimes}} B^{\ast\ast}$, where the second map is obtained by universality. 
Regard $A^{\ast\ast}\odot B^{\ast\ast}$ as a subset of $(A\otimes B)^{\ast\ast}$ via this map. By $a\otimes b\in A^{\ast\ast}\odot B^{\ast\ast}$, we mean an element of the form of $(\displaystyle\lim_{\lambda}a_{\lambda}\otimes 1)(\lim_{\mu}1\otimes b_{\mu})$ for some $a_{\lambda}\in A$ and $b_{\mu}\in B$ with $\displaystyle\lim_{\lambda}a_{\lambda}=a$ and $\displaystyle\lim_{\mu}b_{\mu}=b$ (the limit is taken for your favorite weak topology). 
Then, if we choose $a_{\lambda}\in A$ and $ b_{\lambda}\in B$ as bounded nets converging in the $\sigma$-$\mathrm{strong}^{\ast}$ topology, we can write $a\otimes b=\displaystyle\lim_{\lambda}a_{\lambda}\otimes b_{\lambda}$. This representation is useful and it is easy to check that we can define a similar representation for non-unital $C^*$-algebras $A,B$ and for the inclusion $X^{\ast\ast}\odot Y^{\ast\ast} \subset (X\otimes Y)^{\ast\ast}$, where $X\subset A$ and $Y\subset B$ are operator spaces.

Now, we define the following map $\Theta_{X}$ for any operator space $X$, which plays an important role to study weak exactness. Let $A$ be a $C^{\ast}$-algebra and $M$ be a von Neumann algebra containing $A$ such that $A\subset M$ is $\sigma$-weakly dense. For any $C^{\ast}$-algebra $B$, consider a (non-unital) injective $\ast$-homomorphism
\begin{equation*}
\Theta_{B}\colon A\odot B^{\ast\ast}\subset M\odot B^{\ast\ast} \simeq (z_{M}A^{\ast\ast}\odot B^{\ast\ast}) \subset (A^{\ast\ast}\odot B^{\ast\ast})\longrightarrow (A\otimes B)^{\ast\ast}.
\end{equation*}
Then, by the observation in the paragraph above, the map is of the form $\Theta_{B}(a\otimes b)=z_{M}a\otimes b=\displaystyle\lim_{\lambda}z_{\lambda}a\otimes b_{\lambda}$ for some bounded nets $z_{\lambda}\in A$ and $b_{\lambda}\in B$ satisfying $\displaystyle\lim_{\lambda} z_{\lambda}=z_{M}$ and $\displaystyle\lim_{\lambda} b_{\lambda}=b$.
The map $\Theta_{B}$ is bounded on $A\odot B$. Indeed, for any $a_{k}\in A$ and $ b_{k}\in B$ ($k=1,2,\ldots,n$), we have
\begin{eqnarray*}
\|\Theta_{B}(\sum_{k=1}^{n} a_{k}\otimes b_{k})\|_{(A\otimes B)^{\ast\ast}}
&=& \|\lim_{\lambda} (z_{\lambda}\otimes 1)(\sum_{k} a_{k}\otimes b_{k})\|_{(A\otimes B)^{\ast\ast}} \\
&\leq& \liminf_{\lambda}\| (z_{\lambda}\otimes 1)(\sum_{k} a_{k}\otimes b_{k})\|_{(A\otimes B)^{\ast\ast}} \\
&\leq& \| \sum_{k} a_{k}\otimes b_{k}\|_{(A\otimes B)^{\ast\ast}} \\
&=& \| \sum_{k} a_{k}\otimes b_{k}\|_{(A\otimes B)}.
\end{eqnarray*}
Therefore, $\Theta_{B}$ is isometric on $A\otimes B$ by minimality of $\otimes$. Finally, for any closed subspace $X\subset B$, define the map $\Theta_{X}$ by
\begin{equation*}
\Theta_{X}:=\Theta_{B}\mid_{A\odot X^{\ast\ast}} \colon A\odot X^{\ast\ast} \longrightarrow (A\otimes X)^{\ast\ast} \subset (A\otimes B)^{\ast\ast}.
\end{equation*}
It is easy to check that the image of $\Theta_{X}$ is really contained in $(A\otimes X)^{\ast\ast}$ and $\Theta_{X}$ is isometric on $A\otimes X$.

\subsection{\bf Exactness and weak exactness}

For exactness and weak exactness, we refer the reader to $\cite{Wa}$ and $\cite{BO}$.

A $C^{\ast}$-algebra $A$ is said to be $\it exact$ if $0\rightarrow A\otimes J\rightarrow A\otimes B\rightarrow A\otimes (B/J)\rightarrow 0$ is exact for any short exact sequence $0\rightarrow J\rightarrow B\rightarrow B/J\rightarrow 0$. 
When $A$ is a von Neumann algebra, this definition does not reflect its $\sigma$-weak topology, and so we have to modify it to define exactness for von Neumann algebras appropriately.

It is well known that a $C^{\ast}$-algebra $A$ is exact if and only if the natural map $(A\otimes B)/(A\otimes J)\rightarrow A\otimes(B/J)$ is isomorphic for any $J\lhd B$. This condition is easily translated as follows.
\begin{itemize}
\item For any $J\lhd B$ and any $\ast$-representation $\pi \colon A\otimes B \rightarrow \mathbb{B}(K)$ with $A\otimes J \subset \ker\pi$, the induced $\ast$-representation $\tilde{\pi} \colon A\odot (B/J) \rightarrow \mathbb{B}(K)$ is min-continuous.
\end{itemize}
In the case where $A$ is a von Neumann algebra, we can reflect the $\sigma$-weak topology of $A$ by restricting above $\pi$ to be normal on $A\otimes \mathbb{C}$.

\begin{Def}[{\cite{wk ex K}}]\upshape\label{wk ex}
\ Let $M$ be a von Neumann algebra. We say $M$ is $weakly$ $exact$ if for any $J\lhd B$ and any $\ast$-representation $\pi \colon M\otimes B \rightarrow \mathbb{B}(K)$ with $A\otimes J \subset\ker\pi$ which is normal on $M\otimes \mathbb{C}$, the induced $\ast$-representation $\tilde{\pi} \colon M\odot (B/J) \rightarrow \mathbb{B}(K)$ is min-continuous.
\end{Def}
We leave it to the reader to verify that it is impossible to interpret this definition in terms of exact sequence. Kirchberg proved that a von Neumann algebra $M$ is weakly exact if it contains a $\sigma$-weakly dense exact $C^{\ast}$-algebra. This follows from a deep fact: exactness is equivalent to property $C$. Here recall a $C^{\ast}$-algebra $A$ has $\it property$ $C$ if the inclusion $A^{\ast\ast}\odot B^{\ast\ast}\hookrightarrow (A\otimes B)^{\ast\ast}$ is isometric on $A^{\ast\ast}\otimes B^{\ast\ast}$ for any $C^{\ast}$-algebra $B$.

Next, we consider property $C'$ and a similar condition for von Neumann algebras. Recall a $C^{\ast}$-algebra $A$ has $\it property$ $C'$ if the inclusion $A\odot B^{\ast\ast}\hookrightarrow (A\otimes B)^{\ast\ast}$ is min-continuous for any $C^{\ast}$-algebra $B$. This condition is equivalent to exactness and Ozawa generalized this equivalence in the von Neumann algebra context ($\cite[\rm Proposition\ 4]{wk ex O}$).
More precisely, he proved that a von Neumann algebra $M$ is weakly exact if and only if for any unital $C^{\ast}$-algebra $B$ and any $\ast$-representation $\pi \colon M\otimes B \rightarrow \mathbb{B}(K)$ which is normal on $M\otimes \mathbb{C}$, the induced $\ast$-representation $\tilde{\pi}\colon M\odot B^{\ast\ast} \rightarrow \mathbb{B}(K)$ is min-continuous.

Exactness for operator spaces can be defined in the same way as that for $C^{\ast}$-algebras, but we use a different manner for our convenience. A finite dimensional operator space $E$ is said to be $\it exact$ (more precisely, 1-$\it exact$) if for any $\epsilon>0$, there exist $n\in\mathbb{N}$, an operator space $F\subset\mathbb{M}_{n}$, and an invertible c.b. map $\phi$ from $E$ onto $F$ such that $\|\phi\|_{{\rm cb}}\|\phi^{-1}\|_{{\rm cb}}\leq1+\epsilon$. 
We say an operator space $X$ is $\it exact$ if any finite dimensional subspace of $X$ is exact. This condition is still equivalent to property $C$.

Now, we prove that the direct limit $\displaystyle\varinjlim_{n}(X_{n},\phi_{n})$ defined in the Subsection $\ref{op sp}$ is exact (keep the setting in Subsection $\ref{op sp}$). For this, we observe that $S$ is exact if it has a norm dense subspace satisfying any finite dimensional subspace is exact (exactness follows from this). Indeed, $S_{0}\subset S$ works since the following map is completely isometric:
\begin{equation*}
\mathbb{M}_{k(1)}\oplus \cdots\mathbb{M}_{k(m-1)}\oplus X_{m}\simeq \{ (x_{n})_{n}\in S_{0}\mid x_{m}\in X_{m} \ {\rm and}\  x_{n+1}=\phi_{n}(x_{n}) \ {\rm for\ any} \  n\geq m\}.
\end{equation*}
Notice that the left hand side is contained in a finite dimensional $C^{\ast}$-algebra and that any finite dimensional subspace of $S_{0}$ is contained in the right hand side for  sufficiently large $m\in\mathbb{N}$.
Finally, our claim ends with the following lemma.
\begin{Lem}
Let $A$ be a unital $C^{\ast}$-algebra, $J\subset A$ be a closed ideal, and $X\subset A$ be a closed subspace containing $J$. If $X$ is exact, then $\pi(X)$ is exact, where $\pi \colon A\rightarrow A/J$ is the quotient map. 
\end{Lem}
\begin{proof}
Take the central projection $z\in A^{\ast\ast}$ with $zA^{\ast\ast}\simeq J^{\ast\ast}$. Then, since $zx\in J^{\ast\ast}\subset X^{\ast\ast}$ for any $x\in X^{\ast\ast}$, we have
\begin{equation*}
(1-z)X^{\ast\ast}:=\{ (1-z)x \in A^{\ast\ast} \mid x\in X^{\ast\ast} \}
=\{ (x-zx \in A^{\ast\ast} \mid x\in X^{\ast\ast} \} \subset X^{\ast\ast},
\end{equation*}
(observe that we do not assume $1\in X$). For any $C^{\ast}$-algebra $B$, we have the following commutative diagram:
\[
\begin{CD}
\pi(A)^{\ast\ast}\odot B^{\ast\ast} @>>> (\pi(A)\otimes B)^{\ast\ast}\\
@VVV @AA{(\pi\otimes \mathrm{id}_{B})^{\ast\ast}}A\\
(1-z)A^{\ast\ast}\odot B^{\ast\ast} @>>> (A\otimes B)^{\ast\ast},
\end{CD}
\]
where these maps are defined in the previous subsection. Then, the restriction of the diagram gives the following new diagram:
\[
\begin{CD}
\pi(X)^{\ast\ast}\odot B^{\ast\ast} @>>> (\pi(X)\otimes B)^{\ast\ast}\\
@VVV @AA{(\pi\otimes \mathrm{id}_{B})^{\ast\ast}}A\\
(1-z)X^{\ast\ast}\odot B^{\ast\ast} @>>> (X\otimes B)^{\ast\ast}.
\end{CD}
\]
Now the bottom map is isometric on $(1-z)X^{\ast\ast}\otimes B^{\ast\ast}$ by assumption, and hence the top map is also isometric. This means $\pi(X)$ has property $C$.
\end{proof}

Thus, we proved that $\displaystyle\varinjlim_{n}(X_{n},\phi_{n})$ is exact if each $X_{n}$ is contained in a finite dimensional $C^{\ast}$-algebra completely isometrically.

\subsection{\bf Crossed products and amenable actions}\label{cros}

Let $A\subset \mathbb{B}(H)$ be a $C^{\ast}$-algebra and $\Gamma$ be a discrete group acting on $A$. Consider the following maps:
\begin{alignat*}{3}
\pi &\colon A \longrightarrow \mathbb{B}(H\otimes \ell^{2}(\Gamma))\ &;& \ a \longmapsto \sum_{g\in \Gamma}\alpha_{g^{-1}}(a)\otimes e_{g,g},\\
u &\colon \Gamma \longrightarrow {\cal U}(H\otimes \ell^{2}(\Gamma))\ &;& \ g \longmapsto 1\otimes \lambda_{g},
\end{alignat*}
where $\alpha$ is the $\Gamma$-action, $e_{g,g}$ is the orthogonal projection to $\delta_{g}$, and $\lambda$ is the left regular representation. Define the $\it reduced$ $\it crossed$ $\it product$ $C^{\ast}$-$\it algebra$ of the triple ($A,\Gamma,\alpha$) by
\begin{equation*}
A\rtimes_{r} \Gamma := C^{\ast}(\pi(A),u(\Gamma))\subset \mathbb{B}(H\otimes \ell^{2}(\Gamma)).
\end{equation*}
For a von Neumann algebra $A=:M\subset \mathbb{B}(H)$, define the $\it crossed$ $\it product$ $\it von$ $\it Neumann$ $\it algebra$ by
\begin{equation*}
M\mathbin{\bar{\rtimes}} \Gamma := W^{\ast}(\pi(M),u(\Gamma))\subset \mathbb{B}(H\otimes \ell^{2}(\Gamma)).
\end{equation*}
In the case where $A=\mathbb{C}$, $\mathbb{C}\rtimes_{r} \Gamma$ (respectively, $\mathbb{C}\mathbin{\bar{\rtimes}} \Gamma$) is denoted by $C_{\lambda}^{\ast}(\Gamma)$ (respectively, $L\Gamma$).

The crossed product $M\mathbin{\bar{\rtimes}}\Gamma$ has the canonical faithful normal conditional expectation $E$ onto $M$ defined by $M\mathbin{\bar{\rtimes}}\Gamma\ni x\mapsto (1\otimes e_{e,e})x(1\otimes e_{e,e})\in M$. For a faithful semifinite normal weight $\omega$ on $M$, $\widetilde{\omega}:=\omega\circ E$ is a faithful semifinite normal weight on $M\mathbin{\bar{\rtimes}}\Gamma$. 
Since it is the dual weight of $\omega$ $\cite[\rm Definition\ X.1.16]{Tak 2}$, the GNS-representation of $\widetilde{\omega}$ is isomorphic to $L^{2}(M,\omega)\otimes \ell^{2}(\Gamma)$ as a standard representation and Tomita's conjugation here is given by  $\xi\otimes \delta_{g}\mapsto v_{g}J_{M}\xi\otimes\delta_{g^{-1}}$ ($\cite[\rm Lemma\ X.1.13]{Tak 2}$). Here $J_{M}$ is Tomota's conjugation on $L^{2}(M,\omega)$ and $v_{g}$ is the unique unitary element which satisfies $\alpha_{g}=\mathrm{Ad}v_{g}$ and preserves the standard representation structure of $L^{2}(M,\omega)$ ($\cite[\rm Theorem\ IX.1.14]{Tak 2}$). By simple calculations, it holds
\begin{alignat*}{5}
&J(1\otimes \lambda_{s})J=v_{s}\otimes \rho_{s}& \hspace{1.5em} &(s\in\Gamma),&\\
&J\pi(x)J=J_{M}xJ_{M}\otimes 1&\hspace{1.5em} &(x\in M),&
\end{alignat*}
where $\rho_{s}\delta_{g}:=\delta_{gs^{-1}}$ is the right translation.

To generalize this reduced crossed product construction, we consider arbitrary $\ast$-representation $\pi \colon A \rightarrow \mathbb{B}(K)$ and unitary representation $u \colon \Gamma \rightarrow {\cal U}(K)$. We say $(\pi , u)$ is a $\it covariant$ $\it representation$ if it satisfies $\pi(\alpha_{g}(a))=u_{g}\pi(a)u_{g}^{\ast}$ for any $a\in A$ and $g\in\Gamma$. 
Equivalently, we can define a covariant representation as a $\ast$-representation of $C_{c}(\Gamma,A)$. Here the $\ast$-algebra $C_{c}(\Gamma,A)$ is defined by the set of all compact support maps from $\Gamma$ to $A$ with the involution $(as)^{\ast}:=\alpha_{s^{-1}}(a)s$ and the multiplication 
$(as)(bt):=a\alpha_{s}(b)st$ for $a,b\in A$ and $s,t\in\Gamma$. A $C^{\ast}$-norm is defined on $C_{c}(\Gamma,A)$ by
\begin{equation*}
\|x\|:=\sup \{ \|\phi(x)\| \mid \phi\ {\rm is\ a\ cyclic}\ \ast{\rm \mathchar`-representation\ of}\ C_{c}(\Gamma,A)\}
\end{equation*}
for $x\in C_{c}(\Gamma,A)$. Define the $\it full$ $\it crossed$ $\it product$ $A\rtimes \Gamma$  by the completion of $C_{c}(\Gamma,A)$ with respect to this norm. In the case where $A=\mathbb{C}$, $\mathbb{C}\rtimes \Gamma$ is denoted by $C^{\ast}(\Gamma)$.

For any amenable group $\Gamma$, it is well known that $C_{\lambda}^{\ast}(\Gamma)$ has many useful properties such as $C_{\lambda}^{\ast}(\Gamma)=C^{\ast}(\Gamma)$. 
Amenability of groups can be translated to group actions and we briefly recall it (or see  $\cite{AR}$, $\cite{BO}$). Let $\Gamma$ be a discrete group and $X$ be a compact Hausdorff space. Assume $\Gamma$ acts on $X$ (or equivalently, $\Gamma$ acts on $C(X)$). 
We say the action is $\it amenable$ if there exists a net of continuous maps $m_{i}\colon X\rightarrow \mathrm{Prob}(\Gamma)\subset \ell^{1}(\Gamma)$ such that 
$\displaystyle\lim_{i}\sup_{x\in X}\|s\cdot m_{i}^{x}-m_{i}^{s\cdot x}\|_{1}=0$ for any $s\in \Gamma$. Here $s\cdot m_{i}^{x}(g):=m_{i}^{x}(s^{-1}g)$ and the topology of $\mathrm{Prob}(\Gamma)$ is the $\rm{weak}^{\ast}$ topology induced from $\ell^{1}(\Gamma)\simeq (\ell^{\infty})_{\ast}$. If $\Gamma$ acts on $X$ amenably, we have $C(X)\rtimes_{r}\Gamma=C(X)\rtimes\Gamma$.

The idea of amenable actions can be generalized to actions on $C^{\ast}$-algebras. An action of $\Gamma$ on a $C^{\ast}$-algebra $A$ is said to be $\it amenable$ if the center of $A$ contains $C(X)$ such that the $\Gamma$-action gives an amenable action on $C(X)$. In this case, we have $A\rtimes_{r}\Gamma=A\rtimes\Gamma$.

\subsection{\bf Hilbert modules, Toeplitz--Pimsner algebras and free products}\label{free}

For Hilbert module theory, we refer the reader to Lance's book ($\cite{Lance}$).

Let $A$ be a unital $C^{\ast}$-algebra and $H$ be a $\mathbb{C}$-vector space. Assume $H$ is a right $A$-module (i.e., there exists a ring anti-homomorphism from $A$ to $\mathrm{End}_{\mathbb{C}}(H)$). We say $H$ is a $\it Hilbert$ $A$-$\it module$ if it satisfies following two conditions:
\begin{itemize}
	\item[$(1)$] There exists an $A$-valued map (called an inner product) $\langle\cdot\ ,\cdot\rangle \colon H\times H\rightarrow A$ satisfying following conditions:
	\begin{itemize}
		\item[$(\mathrm{i})$] $\langle\xi,\cdot\rangle$ is $\mathbb{C}$-linear.
		\item[$(\mathrm{ii})$]$\langle\xi,\eta a\rangle=\langle\xi,\eta\rangle a$.
		\item[$(\mathrm{iii})$]$\langle\xi,\eta\rangle^{\ast}=\langle\eta,\xi\rangle$.
		\item[$(\mathrm{iv})$]$0\leq\langle\xi,\xi\rangle$, $\langle\xi,\xi\rangle=0  \Leftrightarrow \xi=0$ $\hspace{1em}(\xi,\eta\in H,\ a\in A)$.
	\end{itemize}
	\item[$(2)$]The space $H$ is complete with respect to the norm $\|\xi\|:=\|\langle\xi,\xi\rangle\|_{A}^{1/2}$, $(\xi\in H)$.
\end{itemize}
We say a map $x \colon H\rightarrow H$ is $\it adjointable$ if there exists a map $x^{\ast} \colon H \rightarrow H$ satisfying $\langle \xi,x\eta\rangle=\langle x^{\ast}\xi,\eta\rangle$ for any $\xi, \eta\in H$. For any adjointable map $x$ on a Hilbert $A$-module $H$, maps $x$ and $x^{\ast}$ are automatically $A$-module homomorphisms. We let $\mathbb{B}(H)$ be the $\ast$-algebra of all adjointable maps on $H$ which is naturally a $C^{\ast}$-algebra with the operator norm.

Let $A$ and $B$ be $C^{\ast}$-algebras and $H$ be a right $B$-module. Let $\pi_{H} \colon A\rightarrow \mathbb{B}(H)$ be a $\ast$-homomorphism and we regard $H$ as a left $A$-module via $\pi_{H}$. We say $H$ is an $A$-$B$ $\it correspondence$ if these actions commute, that is, it holds $(a \xi)b=a(\xi b)$ for $a\in A$, $b\in B$ and $\xi\in H$. 

Let $A,B$ and $C$ be $C^{\ast}$-algebras, $H$ be an $A$-$B$ correspondence with $\pi_{H}$, and $K$ be a $B$-$C$ correspondence with $\pi_{K}$. 
Define a $C$-valued semi inner product on $H\odot K$ by $\langle\xi\otimes\eta,\xi'\otimes\eta'\rangle:=\langle\eta,\pi_{K}(\langle\xi,\xi'\rangle_{H})\eta'\rangle_{K}$ for $\xi, \xi'\in H$ and $\eta,\eta'\in K$, and construct a new Hilbert $C$-module from $H\odot K$ by separation and completion. 
We say the resulting space is the $\it interior$ $\it tensor$ $\it product$ of $H$ and $K$, and denote it by $H\otimes_{B} K$. This space has a left $A$-action defined by $a(\xi\otimes\eta)=a\xi\otimes\eta$ for $a\in A$, $\xi\in H$ and $\eta\in K$, and so it is an $A$-$C$ correspondence.
We note that there exist the following $\ast$-homomorphisms:
\begin{eqnarray*}
\mathbb{B}(H)\longrightarrow \mathbb{B}(H\otimes_{B} K)\ &;&\ x\longmapsto x\otimes1,\\
\pi_{K}(B)' \longrightarrow \mathbb{B}(H\otimes_{B} K)\ &;&\ y\longmapsto 1\otimes y.
\end{eqnarray*}

Next, we define a generalized GNS-representation in a $C^{\ast}$-module setting. Let $A$ be a unital $C^{\ast}$-algebra, $D\subset A$ be a unital $C^{\ast}$-subalgebra, and $E$ be a conditional expectation from $A$ onto $D$. Then, we define an $A$-valued inner product $\langle a,b\rangle:=E(a^{\ast}b)$ on $A$ and construct a new Hilbert $A$-module $L^{2}(A,E)$ from $A$ by separation and completion. 
The left multiplication on $A$ induces the $\ast$-homomorphism $\pi_{E} \colon A\rightarrow \mathbb{B}(L^{2}(A,E))$ with the canonical cyclic vector $\xi_{E}:=\widehat{1_{A}}$ and we denote them by $(\pi_{E},L^{2}(A,E),\xi_{E})$. 
The space $L^{2}(A,E)$ is an $A$-$D$ correspondence with $\pi_{E}$. We say the conditional expectation $E$ is $\it non$-$\it degenerate$ if $\pi_{E}$ is injective. For any state $\omega$  on $D$ and its GNS-representation $L^{2}(D,\omega)$ (which is a $D$-$\mathbb{C}$ correspondence), we have the following natural identification between the Hilbert $\mathbb{C}$-modules (i.e., Hilbert spaces):
\begin{alignat*}{3}
L^{2}(A,E)\otimes_{D}L^{2}(D,\omega)&\xrightarrow{\ \sim\ }& L^{2}(A,\omega\circ E),&\\
\widehat{a}\otimes \widehat{d}\hspace{4em} &\longmapsto& \widehat{ad}.\hspace{2em}&
\end{alignat*}
Summary, we have the following maps:
\begin{alignat*}{4}
A &\xrightarrow{\ \pi_{E}}&\ \mathbb{B}(L^{2}(A,E)) &\longrightarrow&\ \mathbb{B}(L^{2}(A,E)\otimes_{D}L^{2}(D,\omega)) \xrightarrow{\ \sim}\ \mathbb{B}(L^{2}(A,\omega\circ E)),\\
a &\longmapsto& \pi_{E}(a)\hspace{1.7em} &\longmapsto& \pi_{E}(a)\otimes 1. \hspace{13.6em}
\end{alignat*}
Note that the composite map coincides with the GNS-representation of $\omega\circ E$ and the second map is injective if $\omega\circ E$ is non-degenerate. In the case where $D\subset A=:M$ are von Neumann algebras and $E$ and $\omega$ are normal, the inclusion $M\hookrightarrow\mathbb{B}(L^{2}(M,\omega\circ E))$ is normal.

Let $A$ be a unital $C^{\ast}$-algebra and $H$ be an $A$-$A$ correspondence with $\pi_{H}$. Assume $\pi_{H}$ is injective and regard $A\subset \mathbb{B}(H)$. We will define a $C^{\ast}$-algebra from $H$ which is a generalization of the classical Toeplitz algebra due to Pimsner $\cite{Pimsner}$. We first put
\begin{equation*}
H^{\otimes 0}:=A,\  H^{\otimes n}:=\overbrace{H\otimes_{A}\cdots \otimes_{A}H}^{n},\ 
{\cal F}(H):=\bigoplus_{n\geq 0}H^{\otimes n}
\end{equation*}
and define a left $A$-action on ${\cal F}(H)$ by the left multiplication on $H^{\otimes0}$ and  $a(\xi_{1}\otimes\cdots\otimes\xi_{n}):=(a\xi_{1})\otimes\cdots\otimes\xi_{n}$ on $H^{\otimes n}$. Regard $A$ as a subset of $\mathbb{B}({\cal F}(H))$. 
For $\xi\in H$, define an $A$-module homomorphism $T_{\xi} \colon {\cal F}(H)\rightarrow H\otimes_{A}{\cal F}(H)\subset {\cal F}(H)$ by
\begin{eqnarray*}
\left\{
\begin{array}{l}
 H^{\otimes 0} \longrightarrow H\otimes_{A}A\simeq H\ ;\ \widehat{a}\longmapsto \xi\otimes a =\xi a,\\
 H^{\otimes n} \longrightarrow H\otimes_{A}H^{\otimes n}\quad\hspace{0.15em} ; \ \xi_{1}\otimes\cdots\otimes\xi_{n} \longmapsto\xi\otimes\xi_{1}\otimes\cdots\otimes\xi_{n}.
\end{array}
\right.
\end{eqnarray*}
Then, we have a linear map $T \colon H\rightarrow \mathbb{B}(H)$ satisfying $T_{\xi}^{\ast}T_{\eta}=\langle\xi,\eta\rangle$ and $T_{a\xi b}=aT_{\xi}b$ for $a,b\in A$ and $\xi,\eta\in H$. 
Define the $\it Toeplitz$-$\it Pimsner$ $\it algebra$ of $H$ as ${\cal T}(H):=C^{\ast}(A,T_{\xi}\ (\xi\in H))\subset \mathbb{B}({\cal F}(H))$. The algebra ${\cal T}(H)$ has the canonical conditional expectation $E$ onto $A$ whose GNS-representation is isomorphic to the triple (inclusion, ${\cal F}(H), \widehat{1_{A}}$) and 
${\cal T}(H)$ has the universality as following:
\begin{itemize}
\item for any $C^{\ast}$-algebra $B$, any $\ast$-homomorphism $\pi \colon A\rightarrow B$
 and any linear map $\tau \colon H\rightarrow B$ with the relations $\tau(a\xi b)=\pi(a)\tau(\xi)\pi(b)$ and $\tau(\xi)^{\ast}\tau(\eta)=\pi(\langle\xi,\eta\rangle)$ for any $a,b\in A$ and $\xi,\eta\in H$, there exists a $\ast$-homomorphism from ${\cal T}(H)$ to $C^{\ast}(\pi(A),\tau(H))\subset B$ sending $a$ to $\pi(a)$ and $T_{\xi}$ to $\tau(\xi)$ for any $a\in A$ and $\xi\in H$.
\end{itemize}

We mention that ${\cal T}(\mathbb{C})$ is the classical Toeplitz algebra with the vacuum state. In the case where  $A=:M$ is a von Neumann algebra and the left action $M\hookrightarrow \mathbb{B}(H)$ is normal, the inclusion
\begin{equation*}
M\hookrightarrow {\cal T}(M)\subset \mathbb{B}(L^{2}({\cal T}(M),E))\hookrightarrow \mathbb{B}(L^{2}({\cal T}(M),E)\otimes_{M} L^{2}(M,\omega))\simeq \mathbb{B}(L^{2}({\cal T}(M),\omega\circ E))
\end{equation*}
is normal for any faithful normal state $\omega$ on $M$. We will use the inclusion and the von Neumann algebra ${\cal T}(H)''\subset \mathbb{B}(L^{2}({\cal T}(M),\omega\circ E))$ later.

Let $A$ and $B$ be $C^{\ast}$-algebras and $H$ be an $A$-$A$ correspondence. 
Define an $(A\otimes B)$-valued semi inner product on $H\odot B$ by $\langle\xi\otimes\widehat{a},\eta\otimes\widehat{b}\rangle:=\langle\xi,\eta\rangle_{H}\otimes a^{\ast}b\in A\otimes B$ and construct a new Hilbert $(A\otimes B)$-module $H\otimes_{{\rm ext}}B$ by separation and completion. The resulting space is said to be the $\it exterior$ $\it tensor$ $\it product$.
In this case, we naturally have $\mathbb{B}(H)\otimes B\subset \mathbb{B}(H\otimes_{{\rm ext}} B)$ by $(x\otimes a)(\xi\otimes\widehat{b}):=x\xi\otimes\widehat{ab}$ so that $H\otimes_{{\rm ext}}B$ is an $(A\otimes B)$-$(A\otimes B)$ correspondence. It is known ($\cite[\rm Lemma\ 4.6.24]{BO}$) that there exists an isomorphism ${\cal T}(H\otimes_{{\rm ext}} B)\simeq {\cal T}(H)\otimes B$ which extends
\begin{alignat*}{5}
&H\otimes_{{\rm ext}} B&\ \ni&\ \xi\otimes \widehat{b}\longmapsto\mathrm{\xi}\otimes b\ &\in\ &{\cal T}(H)\otimes B,&\\
&\hspace{0.6em}A\otimes B&\ \ni&\ a\otimes b \longmapsto a\otimes b&\in\ &{\cal T}(H)\otimes B.&
\end{alignat*}

Next, we recall amalgamated free products. Let $A_{i}$ $(i\in I)$ be unital $C^{\ast}$-algebras and $D$ be a common unital $C^{\ast}$-subalgebra of all $A_{i}$'s. Assume each $A_{i}$ has a non-degenerate conditional expectation $E_{i}$ onto $D$. 
Denote the GNS-representation of $E_{i}$ by $(\pi_{i},H_{i},\xi_{i})$ and decompose $H_{i}=\xi_{i}D\oplus H_{i}^{o}$ as a Hilbert space, where $H_{i}^{0}:=(\xi_{i}D)^{\perp}$. Note that $H_{i}^{0}$ is naturally a $D$-$D$ correspondence. Define two large $D$-$D$ correspondences by
\begin{eqnarray*}
H:=\Omega D \oplus \bigoplus_{n\geq1}\hspace{0.3em}&\bigoplus_{i_{k}\neq i_{l}\hspace{0.2em} \textrm{for}\hspace{0.2em} k\neq l}& H_{i_{1}}^{0}\otimes_{D}H_{i_{2}}^{0}\cdots \otimes_{D}H_{i_{n}}^{0},\\
H(i):=\Omega D \oplus \bigoplus_{n\geq1}\hspace{0.3em}&\bigoplus_{i\neq i_{1},\hspace{0.2em} i_{k}\neq i_{l}\hspace{0.2em} \textrm{for}\hspace{0.2em} k\neq l}& H_{i_{1}}^{0}\otimes_{D}H_{i_{2}}^{0}\cdots \otimes_{D}H_{i_{n}}^{0},
\end{eqnarray*}
where $\Omega$ is a fixed norm one vector. Next, define unitary operators $U_{i}$ by
\begin{eqnarray*}
U_{i}\colon H_{i}\otimes_{D}H(i)&=&(\xi_{i}D\oplus H_{i}^{0})\otimes_{D}H(i)\\
&\simeq& (\xi_{i}D\otimes_{D}H(i))\oplus(H_{i}^{0}\otimes_{D}H(i))\\
&\simeq& H(i)\oplus(H_{i}^{0}\otimes_{D}H(i))\\
&=& H.
\end{eqnarray*}
We define a left $A_{i}$-action on $H$ by the following injective $\ast$-homomorphism:
\begin{alignat*}{9}
\lambda_{i}\colon&\ A_{i}&&\xrightarrow{\ \pi_{i}\ }&\mathbb{B}(H_{i})&\longrightarrow& \mathbb{B}(H_{i}\otimes_{D}H(i))&\xrightarrow{\mathrm{Ad}U_{i}}&\mathbb{B}(H)\hspace{3em}\\
&\ a&&\longmapsto& \pi_{i}(a)\ &\longmapsto& \pi_{i}(a)\otimes1\hspace{1.5em} &\longmapsto& U_{i}(\pi_{i}(a)\otimes1)U_{i}^{\ast}.
\end{alignat*}
Finally, we define the $\it amalgamated$ $\it free$ $\it product$ $C^{\ast}$-$\it algebra$ of $(A_{i},E_{i})_{i}$ as $C^{\ast}(\lambda_{i}(A_{i}),i\in I)\subset \mathbb{B}(H)$ and we denote it by $\ast_{i\in I}(A_{i},E_{i})$ or $\ast_{D}(A_{i},E_{i})$. We often regard each $A_{i}$ as a subset of $\ast_{i\in I}(A_{i},E_{i})$.

The algebra $\ast_{i\in I}(A_{i},E_{i})$ has the canonical conditional expectation onto $D$ defined by $E(a):=\langle\Omega,a\Omega\rangle$ and its GNS-representation is isomorphic to the triple (inclusion,$H,\Omega)$. More generally, $\ast_{i\in I}(A_{i},E_{i})$ has the canonical conditional expectation onto $\ast_{i\in J}(A_{i},E_{i})$ for any $J\subset I$ and the composite map $\ast_{i\in I}(A_{i},E_{i})\rightarrow \ast_{i\in J}(A_{i},E_{i})\rightarrow D$ coincides with the above $E$, where the second map is the canonical conditional expectation on $\ast_{i\in J}(A_{i},E_{i})$. The restriction of $E$ on $A_{i}$ coincides with $E_{i}$ and $E$ is free in the following sense:
\begin{equation*}
E(a_{1}\cdots a_{n})=0 \quad (a_{i}\in \ker(E_{j_{i}})\subset A_{j_{i}},\ j_{1}\neq j_{2} \cdots \neq j_{n}).
\end{equation*}

In the case where $D\subset A_{i}=:M_{i}$ are von Neumann algebras and each $E_{i}$ is normal, the inclusion
\begin{equation*}
M_{i}\hookrightarrow M\subset \mathbb{B}(L^{2}(M,E))\hookrightarrow \mathbb{B}(L^{2}(M,E)\otimes_{D} L^{2}(D,\omega))\simeq \mathbb{B}(L^{2}(M,\omega\circ E))
\end{equation*}
is normal for any faithful normal state $\omega$ on $M$, where $M=\ast_{i\in I}(M_{i},E_{i})$. Define the $\it amalgamated$ $\it free$ $\it product$ $\it von$ $\it Neumann$ $\it algebra$ of $(M_{i},E_{i})_{i}$ as $M''\subset \mathbb{B}(L^{2}(M,\omega\circ E))$ and denote it by $\bar{\ast}_{i\in I}(M_{i},E_{i})$ or $\bar{\ast}_{D}(M_{i},E_{i})$. Then, the canonical conditional expectation $E$ can be defined by the same manner as above.

Let $\ast_{D}(A_{i},E_{i})$ and $\ast_{D}(B_{i},F_{i})$ be amalgamated free products. Then, for  u.c.p. maps $\phi_{i} \colon A_{i} \rightarrow B_{i}$ with $\phi_{i}\mid_{D}=\mathrm{id}_{D}$ and $E_{i}=F_{i}\circ \phi_{i}$, the free product map $\ast_{i}\phi_{i}$ from $\ast_{D}(A_{i},E_{i})$ to $\ast_{D}(B_{i},F_{i})$ can be defined by the relation $\ast_{i}\phi_{i}(a_{1}a_{2}\cdots a_{n})=\phi_{i_{1}}(a_{1})\phi_{i_{2}}(a_{2})\cdots\phi_{i_{n}}(a_{n})$ for $a_{j}\in A_{i_{j}}$ and $i_{1}\neq i_{2}\cdots\neq i_{n}$ $\cite{Dykema}$. In the case where each $A_{i},B_{i}$ is a von Neumann algebra and each $\phi_{i}$ is normal, $\ast_{i}\phi_{i}$ can be defined as a normal map.

\subsection{\bf Dimension of right modules on finite von Neumann algebras}\label{dim}

Let $M,N$ be von Neumann algebras and $H$ be a Hilbert space. We say $H$ is a $\it right$ (respectively, $\it left$) $M$-$\it module$ if there exists a unital $\ast$-anti-homomorphism $\rho$ (respectively, a $\ast$-homomorphism $\pi$) from $M$ to $\mathbb{B}(H)$. 
We denote these actions by $\rho(x)\xi=\xi x$ (respectively, $\pi(x)\xi=x\xi$) for $x\in M$ and $\xi \in H$. Note the existence of a $\ast$-anti-homomorphism $\rho$ is equivalent to that of a $\ast$-homomorphism $\rho^{\rm op}$ from $M^{\rm op}$ to $\mathbb{B}(H)$, 
where $M^{\rm op}$ is the opposite algebra of $M$ and $\rho^{\rm op}(x)=\rho(x)$ for any $x\in M$. A Hilbert space $H$ is said to be an $M$-$N$-$\it bimodule$ if it is a left $M$-module and is a right $N$-module and these actions commute, that is, $(x\xi)y=x(\xi y)$ for $x\in M$, $y\in N$, and $\xi \in H$. 
For a right (respectively, left) $M$-module $H$, a closed subspace $K\subset H$ is said to be a $\it right$ (respectively, $\it left$) $M$-$\it submodule$ of $H$ if $KM\subset K$ (respectively, $KM\subset K$). Note that a closed subspace $K\subset H$ is a right $M$-submodule if and only if the orthogonal projection to $K$ is contained in $\rho(M)'$, where $\rho$ is the associated $\ast$-anti-homomorphism of $M$.

For a faithful semifinite normal trace state $\tau$ on $M$, the GNS-representation $L^{2}(M,\mathrm{Tr})$ has the canonical right $M$-module structure defined by $\widehat{x}a=\widehat{xa}$\ $(x,a\in M)$, and we denote this right $M$-action by $\rho_{M}$. Since $\rho_{M}$ is faithful, we can regard $M^{\rm op}$ as a subset of $\mathbb{B}(L^{2}(M,\mathrm{Tr}))$ via $\rho_{M}^{\rm op}$. 
It is known that $(M^{\rm op})'=M$.

Let $M$ be a semifinite von Neumann algebra with a faithful semifinite normal trace state $\tau$ and $H$ be a right $M$-module with a $\ast$-anti-homomorphism $\rho \colon M\rightarrow\mathbb{B}(H)$. We will define a kind of dimension of $H$ on $M$. At first, by  $\cite[\rm{Theorem\ IV.5.5}]{Tak 1}$, we decompose $\rho^{\rm op}$ into the following maps:
\begin{alignat*}{5}
&\mathbb{B}(L^{2}) && \mathbb{B}(K\otimes L^{2})&& \mathbb{B}(p(K\otimes L^{2})\  &\xrightarrow{\ \sim\ }&\ \ \mathbb{B}(H)\\
&\ \ \cup&&\hspace{2em}\cup&&\hspace{2.5em}\cup&&\hspace{1.5em}\cup\\
\rho^{{\rm op}} \colon&\ M^{{\rm op}}&\longrightarrow&\ \ 1\otimes M^{{\rm op}}&\longrightarrow&\ p(1\otimes M^{{\rm op}})p\ &\xrightarrow{\ \sim\ }&\ \rho^{{\rm op}}(M^{{\rm op}})=\rho(M)\\
&\ \ x&\longmapsto&\quad 1\otimes x &\longmapsto&\ \ p(1\otimes x)p &\longmapsto&\ \  \rho^{{\rm op}}(x)\ \ \ =\rho(x),
\end{alignat*}
where $L^{2}:=L^{2}(M,\tau)$, $K$ is a separable Hilbert space, and $p$ is a projection in $(\mathbb{C}1\mathbin{\bar{\otimes}}M^{{\rm op}})'=\mathbb{B}(K)\mathbin{\bar{\otimes}}M$. Since the final map is a spacial isomorphism, we have an isometry $p(K\otimes L^{2})\xrightarrow{\sim} H$ as a right $M$-module. 
Since $K\otimes L^{2}$ is isomorphic to $\sum_{n\in\mathbb{N}}\oplus L^{2}$, the Hilbert space $H$ is isomorphic to a right $M$-submodule of a direct sum of $L^{2}(M)$ on which $M$ acts diagonally.

For above $M,\tau,H$, and $p$, we define the $\it dimension$ of $H$ as a right $M$-module by $\widetilde{\mathrm{Tr}}(p)\in [0,\infty]$, where $\widetilde{\mathrm{Tr}}:=\mathrm{Tr}_{\mathbb{B}(K)}\otimes\tau$. 
Note that the dimension depends on the choice of $\tau$ but not of the choice of $p$. Indeed, if $V,W \colon H\hookrightarrow K\otimes L^{2}$ are right $M$-module isometries for a Hilbert space $K$ (possibly non-separable), then we automatically have $WV^{\ast}\in(\mathbb{C}1\mathbin{\bar{\otimes}}M^{{\rm op}})'=\mathbb{B}(K)\mathbin{\bar{\otimes}}M$ and hence
\begin{equation*}
\widetilde{\mathrm{Tr}}(VV^{\ast})=\widetilde{\mathrm{Tr}}(VW^{\ast}WV^{\ast})=
\widetilde{\mathrm{Tr}}((WV^{\ast})^{\ast}WV^{\ast})=
\widetilde{\mathrm{Tr}}(WV^{\ast}(WV^{\ast})^{\ast})=\widetilde{\mathrm{Tr}}(WW^{\ast}).
\end{equation*}

Let $M$ be a semifinite von Neumann algebra with a faithful semifinite normal tracial weight Tr and $p\in M$ be a projection with $\mathrm{Tr}(p)<\infty$. Then, for any projection $q\in M$, the dimension of $qL^{2}(M)p$ on $pMp$ with $\mathrm{Tr}(\cdot)/\mathrm{Tr}(p)$ is smaller than $\mathrm{Tr}(q)/\mathrm{Tr}(p)$. 
To verify this, let $(q_{i})_{i}$ be a maximal family of mutually orthogonal projections in $M$ such that  $q_{i}\leq q$ and $q_{i}\sim p_{i}\leq p$ for some projections $p_{i}\in M$, and observe that 
\begin{equation*}
qL^{2}(M)p=\sum_{i}q_{i}L^{2}(M)p\simeq\sum_{i}\oplus p_{i}L^{2}(M)p=\sum_{i}\oplus p_{i}L^{2}(pMp).
\end{equation*}
Since the dimension of $p_{i}L^{2}(pMp)$ is $\mathrm{Tr}(p_{i})/\mathrm{Tr}(p)=\mathrm{Tr}(q_{i})/\mathrm{Tr}(p)$, the dimension of $qL^{2}(M)p$ is $\sum_{i}\mathrm{Tr}(q_{i})/\mathrm{Tr}(p)=\mathrm{Tr}(\sum_{i}q_{i})/\mathrm{Tr}(p)\leq \mathrm{Tr}(q)/\mathrm{Tr}(p)$.

More generally, for any Hilbert space $K$ and any projection $q\in M\mathbin{\bar{\otimes}}\mathbb{B}(K)$, the dimension of $q(L^{2}(M)p\otimes K)$ on $pMp$ is smaller than $\widetilde{\mathrm{Tr}}(q)/\mathrm{Tr}(p)$, where $\widetilde{\mathrm{Tr}}:=\mathrm{Tr}\otimes\mathrm{Tr}_{\mathbb{B}(K)}$. To see this, choose a maximal family of mutually orthogonal projections $(q_{i})_{i}$ in $M\mathbin{\bar{\otimes}}\mathbb{B}(K)$ such that $q_{i}\leq q$ and $q_{i}\sim p_{i}\leq p\otimes 1$ for some projections $p_{i}\in M\mathbin{\bar{\otimes}}\mathbb{B}(K)$. Then we have a similar isometry
\begin{equation*}
q(L^{2}(M)p\otimes K)=\sum_{i}q_{i}(L^{2}(M)p\otimes K)\simeq \sum_{i}\oplus p_{i}(L^{2}(M)p\otimes K)=\sum_{i}\oplus p_{i}(L^{2}(pMp)\otimes K).
\end{equation*}
Since the dimension of $p_{i}(L^{2}(pMp)\otimes K)$ is $\widetilde{\mathrm{Tr}}(p_{i})/\mathrm{Tr}(p)=\widetilde{\mathrm{Tr}}(q_{i})/\mathrm{Tr}(p)$, the dimension of $q(L^{2}(M)p\otimes K)$ is
 $\sum_{i}\widetilde{\mathrm{Tr}}(q_{i})/\mathrm{Tr}(p)=\mathrm{Tr}(\sum_{i}q_{i})/\mathrm{Tr}(p)\leq \mathrm{Tr}(q)/\mathrm{Tr}(p)$.

\section{\bf Weak Exactness for $C^{\ast}$-algebras}

In this section, we generalize the notion of weak exactness to general $C^{\ast}$-algebras with fixed faithful representations. We then give some fundamental properties of this notion which are quite similar to that of weakly exact von Neumann algebras due to Ozawa.

\subsection{\bf Definition and basic properties}

Let $A$ be a $C^{\ast}$-algebra and $\mathbb{B}(H)$ be a fixed faithful representation of $A$. Although $A$ is not a von Neumann algebra in general, it still has a $\sigma$-weak topology endowed from $\mathbb{B}(H)$. 
Throughout this paper, we assume that $A$ comes together with a $\sigma$-weak topology. Note that,  for any von Neumann algebra $M$ with $A\subset M\subset \mathbb{B}(H)$ (e.g. $M=A''$), the $\sigma$-weak topology of $A$ is induced from only that of $M$.

\begin{Def}\label{def wk ex}\upshape
Let $A$ be a $C^{\ast}$-algebra and $\mathbb{B}(H)$ be a faithful representation of $A$. We say $A$ is $\it weakly$ $\it exact$ in $\mathbb{B}(H)$ 
if for any $J\lhd B$ and any $\ast$-representation $\pi \colon A\otimes B \rightarrow \mathbb{B}(K)$ with $A\otimes J \subset\ker\pi$ which is $\sigma$-weakly continuous on $A\otimes \mathbb{C}$, 
the induced $\ast$-representation $\tilde{\pi} \colon A\odot (B/J) \rightarrow \mathbb{B}(K)$ is min-continuous. More generally, we define weak exactness for any inclusion $A\subset M$, where $M$ is a von Neumann algebra, and say $A$ is $\it weakly$ $\it exact$ in $M$.
\end{Def}
\begin{Rem}\upshape
This definition contains the following two cases:
in the case where $A=M$, our weak exactness is equivalent to Kirchberg's weak exactness for von Neumann algebras; and in the case where $M=\mathbb{B}(H_{u})$, where $H_{u}$ is the universal representation of $A$, our weak exactness is equivalent to exactness of $A$ as a $C^{\ast}$-algebra. We note that any exact $C^{\ast}$-algebra $A$ is weakly exact in any faithful representation of $A$.
\end{Rem}

Since our definition is quite similar to Kirchberg's one, Ozawa's characterization ($\cite[\rm Theorem\ 2]{wk ex O}$) can be easily imported to our situation.
\begin{Thm}\label{we ex thm}
Let $A$ be a $C^{\ast}$-algebra and $\mathbb{B}(H)$ be a faithful representation of $A$.
\begin{itemize}
	\item[$(1)$] The following are equivalent.
	\begin{itemize}
		\item[$(\rm{i})$]The $C^*$-algebra $A$ is weakly exact in $\mathbb{B}(H)$.
		\item[$(\rm{ii})$]For any unital $C^{\ast}$-algebra $B$ and any $\ast$-representation $\pi \colon A\otimes B \rightarrow \mathbb{B}(K)$ which is $\sigma$-weakly continuous on $A\otimes \mathbb{C}$, the induced $\ast$-representation $\tilde{\pi}\colon A\odot B^{\ast\ast} \rightarrow \mathbb{B}(K)$ is min-continuous.
		\item[$(\rm{iii})$]For any operator space $X$, the map $\Theta_{X}\colon A\odot X^{\ast\ast} \rightarrow (A\otimes X)^{\ast\ast}$ is isometric, where $\Theta_{X}$ is as in Subsection $\ref{**}$.
		\item[$(\rm{iv})$]For any operator space $X$ and any finite rank c.c. map $\phi\colon X^{\ast} \rightarrow A$, there exists a net $(\phi_{i})$ of $weak^{\ast}$ continuous finite rank c.c. maps $\phi_{i}\colon X^{\ast} \rightarrow A$ which converges to $\phi$ in the point $\sigma$-weak topology.
		\item[$(\rm{v})$]For any finite dimensional operator space $E\subset A$, there exist nets of c.c. maps $\phi_{i}\colon E \rightarrow \mathbb{M}_{n(i)}$ and $\psi_{i} \colon \phi_{i}(E) \rightarrow A$ such that the net $(\psi_{i} \circ \phi_{i})$ converges to $\mathrm{id}_{E}$ in the point $\sigma$-weak topology.
	\end{itemize}
	\item[$(2)$]If $A$ is weakly exact in $\mathbb{B}(H)$ and $H$ is a separable Hilbert space, then there exist a norm-separable exact operator space $S$ and normal c.c. maps $\phi \colon M \rightarrow S^{\ast\ast}$ and $\psi \colon S^{\ast\ast} \rightarrow M$ such that $\psi\circ\phi=\mathrm{id}_{M}$, where $M$ is the $\sigma$-weak closure of $A$.
\end{itemize}
\end{Thm}
\begin{proof}
Except for $(\rm{v})\Rightarrow (\rm{i})$ in (1), the proof of this theorem is almost same as Ozawa's one, and hence is omitted (use first three subsections in Section $\ref{pre}$, if necessary). 
To see $(\rm{v})\Rightarrow (\rm{i})$, let $B,J,\pi$, and $\mathbb{B}(K)$ be as in Definition $\ref{def wk ex}$. Take any $z\in A\odot (B/J)$ and choose a finite dimensional operator space $E\subset A$ such that $z\in E\odot (B/J)$. Then we can find nets of c.c. maps $\phi_{i}$ and $\psi_{i}$ for $E$ as in $(\mathrm{v})$. Put $E_{i}:=\phi_{i}(E)$. Since $\mathbb{M}_{n(i)}$ is exact, $E_{i}$ is 1-exact as an operator space, and so we have the canonical isometry
\begin{equation*}
\frac{E_{i}\otimes B}{E_{i}\otimes J}\simeq E_{i}\otimes (B/J).
\end{equation*}
Now, using this identification, the composite map  $\tilde{\pi}\circ(\psi_{i}\otimes\mathrm{id})$ coincides with the induced map on $(E_{i}\otimes B)/(E_{i}\otimes J)$ from $\pi\circ (\psi_{i}\otimes \mathrm{id})$, and so is bounded. 
Since $\tilde{\pi}\circ((\psi_{i}\circ\phi_{i})\otimes \mathrm{id})(z)$ converges to $\tilde{\pi}(z)$ $\sigma$-weakly and $\tilde{\pi}\circ((\psi_{i}\circ\phi_{i})\otimes \mathrm{id})$ is contraction, we have
\begin{equation*}
\|\tilde{\pi}(z)\| \leq \displaystyle\liminf_{i}\| \tilde{\pi}\circ((\psi_{i}\circ\phi_{i})\otimes \mathrm{id})(z)\| \leq \|z\|.
\end{equation*}
This implies boundedness of $\tilde{\pi}$.
\end{proof}

\begin{Rem}\upshape\label{rem v'}
By the proof of (v)$\Rightarrow$(i) in (1), the conditions in (1) above is equivalent to the following condition.
\begin{itemize}
	\item[$(\rm{v}')$]There exists a norm dense subspace $A_{0}\subset A$ such that for any finite-dimensional operator space of $A_{0}$, a family of c.c. maps $\phi_{i}$ and $\psi_{i}$ as in condition (v) above exist.
\end{itemize}
\end{Rem}

For any von Neumann algebra $M$ with separable predual, the existence of $S$ above is equivalent to weak exactness of $M$ by Ozawa's original theorem. Hence we have the main theorem in this section.

\begin{Cor}\label{wk ex main}
Let $A\subset \mathbb{B}(H)$ be a $C^{\ast}$-algebra with a fixed faithful representation.  If $H$ is separable and $A$ is weakly exact in $\mathbb{B}(H)$, then the $\sigma$-weak closure of $A$ is weakly exact.
\end{Cor}

As we mentioned, a von Neumann algebra is weakly exact if it contains an exact $C^{\ast}$-algebra which is $\sigma$-weakly dense in the von Neumann algebra. Here we give a partial converse but is not useful.
\begin{Cor}
If M is weakly exact von Neumann algebra with separable predual, then M has a norm-separable $\sigma$-weakly dense $C^{\ast}$-subalgebra which is weakly exact in M.
\end{Cor}
\begin{proof}
Take $S$ and $\psi$ in the second statement in the previous theorem and put $A:=C^{\ast}(\psi(S))$. Then $A$ does the work (note that $\psi(S)\subset M$ is $\sigma$-weakly dense).
\end{proof}

For group von Neumann algebras, our weak exactness is equivalent to exactness of groups in the following sense.
\begin{Cor}
For any discrete group $\Gamma$, the following are equivalent.
\begin{itemize}
	\item[$(1)$]The group $\Gamma$ is exact $\rm{(}$i.e. $C_{\lambda}^{\ast}(\Gamma)$ is exact$\rm{)}$.
	\item[$(2)$]The $C^*$-algebra $C_{\lambda}^{\ast}(\Gamma)$ is weakly exact in $\mathbb{B}(\ell^{2}(\Gamma))$.
	\item[$(3)$]The von Neumann algebra $L\Gamma$ is weakly exact.
\end{itemize}
\end{Cor}
\begin{proof}
We know (1)$\Rightarrow$(2), and (2)$\Rightarrow$(3) is from Corollary $\ref{wk ex main}$ for countable $\Gamma$. The implication (3)$\Rightarrow$(1) is due to Ozawa $\cite[\rm Corollary\ 3]{wk ex O}$ and this proof also shows the implication (2)$\Rightarrow$(1) for general $\Gamma$.
\end{proof}

\section{\bf Some Permanence Properties for Weak Exactness}

In this section, we prove some permanence properties which are inspired from those of exact groups. Our weak exactness will be used to prove them for both $C^{\ast}$-algebras and von Neumann algebras.

\subsection{\bf Tensor products, crossed products, and increasing unions}

Exactness of groups is preserved under extensions and increasing unions. Corresponding results for von Neumann algebras will be studied.

Before proceeding, we recall Takesaki's conditional expectation theorem $\cite[\rm{Theorem\ IX.4.2}]{Tak 2}$:
\begin{itemize}
\item for von Neumann algebras $N\subset M$ (with common unit) and a faithful normal state $\omega$ on $M$, $M$ has a conditional expectation onto $N$ preserving $\omega$ if and only if $\sigma^{\omega}_{t}(N)=N$ for any $t\in\mathbb{R}$, where $\sigma^{\omega}$ is the modular automorphism group of $\omega$. 
This $\omega$-preserving conditional expectation is unique and is automatically faithful and normal.
\end{itemize}
Note that this conditional expectation $E$ satisfies $\omega(xy)=\omega(E(x)y)$ for any $x\in M$ and $y\in N$, and $E$ is determined by the relation.

Here we give a simple corollary which we need in this section.
\begin{Lem}\label{Tak lem}
For any von Neumann algebra M, there exists a net of normal (possibly non-unital) conditional expectations $(E_{i})_{i}$ on M such that each image of $E_{i}$ has separable predual and $E_{i}\rightarrow \mathrm{id}_{M}$ in the point $\sigma$-weak topology.
\end{Lem}
\begin{proof}
Since $1_{M}$ is approximated by $\sigma$-finite projections $\sigma$-weakly, we may assume $M$ is $\sigma$-finite. Fix a faithful normal state $\omega$ and take any sequence $(a_{n})_{n}$ in $M$. Let $M_{0}$ be a unital von Neumann subalgebra of $M$  generated by $\sigma^{\omega}_{t}(a_{n})$, where $n\in\mathbb{N}$ and $t\in\mathbb{Q}$.
Then, since $\sigma_{t}^{\sigma}(M_{0})=M_{0}$ for any $t\in \mathbb{R}$, there exists a faithful normal conditional expectation from $M$ onto $M_{0}$.
\end{proof}

We first study increasing unions.
\begin{Pro}\label{inc uni}
Let $M$ be a von Neumann algebra and $(M_{i})$ be an increasing net of von Neumann subalgebras of $M$ with $M=(\bigcup_{i} M_{i})''$. If each $M_{i}$ is weakly exact, then the norm closure of \ $\bigcup_{i} M_{i}$ is weakly exact in $M$. In particular, $M$ is weakly exact if it has separable predual.
\end{Pro}
\begin{proof}
This is obvious since $\bigcup_{i} M_{i}$ satisfies ($\rm{v}'$) in Remark \ref{rem v'} as a subspace of the norm closure of $\bigcup_{i} M_{i}$. The last statement follows from Corollary $\ref{wk ex main}$.
\end{proof}
Observe that it is impossible to prove the last statement in this proposition if we do not know Corollary $\ref{wk ex main}$. This is a similar phenomena to the case of semidiscreteness which we mentioned in the first part of the paper.

Next, we study tensor products and crossed products which correspond to extensions of groups. The case of tensor products was already solved ($\cite[\rm Corollary\ 14.2.5]{BO}$), but we prove it here as a corollary to Theorem $\ref{we ex thm}$.

\begin{Pro}\label{ten pro}
If $M_{1}$ and $M_{2}$ are weakly exact von Neumann algebras, then $M_{1}\otimes M_{2}$ is weakly exact in $M_{1}\mathbin{\bar{\otimes}} M_{2}$. In particular, $M_{1}\mathbin{\bar{\otimes}} M_{2}$ is weakly exact.
\end{Pro}
\begin{proof}
It is easy to check that the norm dense subspace $M_{1}\odot M_{2}\subset M_{1}\otimes M_{2}$ satisfies ($\rm{v}'$) in Remark \ref{rem v'}, and so $M_{1}\otimes M_{2}$ is weakly exact in $M_{1}\mathbin{\bar{\otimes}} M_{2}$. Hence, $M_{1}\mathbin{\bar{\otimes}} M_{2}$ is weakly exact by Corollary $\ref{wk ex main}$, if $M_{1}$ and $M_{2}$ have separable predual. General case follows from Lemma $\ref{Tak lem}$.
\end{proof}

The case of crossed products is rather difficult than that of tensor products because of its twisted structures. When we prove exactness of crossed products with exact groups and exact $C^{\ast}$-algebras, the full crossed products may be used (see the proof of $\cite[\rm Theorem\ 10.2.9]{BO}$). 
This is because the full crossed products has the useful universality and has exactness in the following sense: for any discrete group $\Gamma$, exact sequence $0\rightarrow J\rightarrow A\rightarrow A/J\rightarrow0$, and any $\Gamma$-action on $A$ with $J$ globally $\Gamma$-invariant, we have the canonical isomorphism 
$(A\rtimes\Gamma)/(J\rtimes\Gamma)\simeq (A/J)\rtimes\Gamma.$
 We would like to imitate the proof in our case, but we cannot do this since we no longer accept exact sequences in our context. However, the universality of the full crossed products is still useful and will be used essentially. 

We first prepare some lemmas. Let $B$ be a unital $C^{\ast}$-algebra, $M$ be a weakly exact von Neumann algebra and $\Gamma$ be a discrete group acting on $M$. Considering trivial $\Gamma$-actions on $B$ and $B^{\ast\ast}$, we construct full crossed products $(B\otimes M)\rtimes \Gamma$ and $(B^{\ast\ast}\otimes M)\rtimes \Gamma$. 
We then recall following maps:
\begin{eqnarray*}
\Theta_{B}&\colon& B^{\ast\ast} \otimes M \longrightarrow (B\otimes M)^{\ast\ast},\\
u&\colon& \Gamma \longrightarrow (B\otimes M)\rtimes \Gamma,
\end{eqnarray*}
where $\Theta_{B}$ is defined in Subsection $\ref{**}$ and $u$ is the canonical inclusion. We note that both images of these maps are contained in $((B\otimes M)\rtimes \Gamma)^{\ast\ast}$.

\begin{Lem}\label{lemma a}
The pair $(\Theta_{B},u)$ is a covariant representation into $((B\otimes M)\rtimes \Gamma)^{\ast\ast}$. Hence, we have a $\ast$-homomorphism $\Theta_{B}\times u$ from $C_{c}(\Gamma,B^{\ast\ast} \otimes M)$ into $((B\otimes M)\rtimes \Gamma)^{\ast\ast}$. 
\end{Lem}
\begin{proof}
Take any $b \in B^{\ast\ast}$, $x\in M$ and write $\displaystyle\Theta_{B}(b\otimes x)=b\otimes xz_{M} =\lim_{\lambda} b_{\lambda}\otimes xz_{\lambda}$ for some $b_{\lambda} \in B$ and  $z_{\lambda}\in M$. Then, we have
\begin{eqnarray*}
	u_{g}\Theta_{B}(b\otimes x)u_{g}^{\ast} &=& \lim u_{g}(b_{\lambda}\otimes xz_{\lambda})u_{g}^{\ast}\\
&=& \lim b_{\lambda}\otimes \alpha_{g}(xz_{\lambda})\\
&=& b\otimes \tilde{\alpha}_{g}(xz)\\
&=& b\otimes \alpha_{g}(x)z\\
&=& \lim  b_{\lambda}\otimes \alpha_{g}(x)z_{\lambda}\\
&=& \Theta_{B}(b\otimes \alpha_{g}(x))\\
&=& \Theta_{B}(\alpha_{g}(b\otimes x)),
\end{eqnarray*}
where $\tilde{\alpha}_{g}$ is the natural extension of $\alpha_{g}$ on $M^{\ast\ast}$ (and we used the fact $\tilde{\alpha}_{g}(z)=z$). This implies $(\Theta_{B},u)$ is a covariant representation of $B^{\ast\ast}\otimes M$ and $\Gamma$.
\end{proof}
We extend the map $\Theta_{B}\times u$ on $(B^{\ast\ast}\otimes M)\rtimes \Gamma$ by universality and we still write it as $\Theta_{B}\times u$.

\begin{Lem}\label{lemma A}
Let $\pi$ be a $\ast$-homomorphism from $(B\otimes M)\rtimes \Gamma\otimes C_{\lambda}^{\ast}(\Gamma)$ to $\mathbb{B}(H)$ which is normal on $\mathbb{C}\otimes M\otimes \mathbb{C}$. 
Assume $\Gamma$ is exact. Then, the induced $\ast$-homomorphism from $C_{c}(\Gamma,B^{\ast\ast}\odot M)\odot C_{\lambda}^{\ast}(\Gamma)$ to $\mathbb{B}(H)$ is continuous with respect to the norm on $(B^{\ast\ast}\otimes M)\rtimes \Gamma\otimes C_{\lambda}^{\ast}(\Gamma)$.
\end{Lem}
\begin{proof}
Since $C_{\lambda}^{\ast}(\Gamma)$ has property $C'$, there exists the extended map
\begin{eqnarray*}
\tilde{\pi} \colon ((B\otimes M)\rtimes \Gamma)^{\ast\ast}\otimes C_{\lambda}^{\ast}(\Gamma) \longrightarrow \mathbb{B}(H).
\end{eqnarray*}
Then, the composite map 
\begin{equation*}
(B^{\ast\ast}\otimes M)\rtimes \Gamma\otimes C_{\lambda}^{\ast}(\Gamma)\xrightarrow{(\Theta_{B}\times u)\otimes \mathrm{id}} ((B\otimes M)\rtimes \Gamma)^{\ast\ast}\otimes C_{\lambda}^{\ast}(\Gamma) \xrightarrow{\ \quad\tilde{\pi}\quad\ } \mathbb{B}(H)
\end{equation*}
does the work. To verify this, we need the relation $\tilde{\pi}\mid_{1\otimes M\otimes 1}(\cdot z_{M})=\pi\mid_{1\otimes M\otimes 1}(\cdot)$ on $M$, and the rest calculation is routine.
\end{proof}

The following lemma holds in a general setting.
\begin{Lem}\label{lemma B}
Let $A$ be a $C^{\ast}$-algebra and $\Gamma$ be a discrete group acting on $A$. Then, there exist the following maps:
\begin{eqnarray*}
\Phi_{A} &\colon& A\rtimes_{r}\Gamma \longrightarrow A\rtimes\Gamma\otimes C_{\lambda}^{\ast}(\Gamma)\ ;\ au_{g} \longmapsto au_{g}\otimes \lambda_{g} \ \quad \ast\hspace{-0.2em}\mathchar`-hom,\\
\Psi_{A} &\colon& A\rtimes\Gamma\otimes C_{\lambda}^{\ast}(\Gamma) \longrightarrow A\rtimes_{r}\Gamma \ ;\ au_{g}\otimes \lambda_{s} \longmapsto au_{g}\delta_{s,g} \ \ u.c.p.,
\end{eqnarray*}
where $a\in A$ and $g,s\in\Gamma$. In particular, $\Psi_{A}\circ \Phi_{A}=\mathrm{id}_{A\rtimes_{r}\Gamma}$.
\end{Lem}
\begin{proof}
See the proof of $\cite[\rm Theorem\ 5.1.10]{BO}$.
\end{proof}

Now, we prove the case of crossed products.
\begin{Pro}\label{cro pro}
Let $M\subset \mathbb{B}(H)$ be a von Neumann algebra and $\Gamma$ be a discrete group which acts on $M$. If $M$ is weakly exact and $\Gamma$ is exact, then $M\rtimes_{r} \Gamma$ is weakly exact in $\mathbb{B}(H\otimes\ell^{2}(\Gamma))$. In particular, $M\mathbin{\bar{\rtimes}}\Gamma$ is weakly exact if $M$ has separable predual.
\end{Pro}
\begin{proof}
Let $\pi \colon B\otimes(M\rtimes_{r}\Gamma)\longrightarrow \mathbb{B}(K)$ be a $\ast$-homomorphism which is $\sigma$-weakly continuous on $\mathbb{C}\otimes (M \rtimes_{r}\Gamma)$. 
We will find a bounded map on $B^{\ast\ast}\otimes(M\rtimes_{r}\Gamma)$ which coincides with the natural extension of $\pi$ on $B^{\ast\ast}\odot(M\rtimes_{r}\Gamma)$.

Using the trivial $\Gamma$-action on $B$, we identify $B\otimes(M\rtimes_{r}\Gamma) \simeq (B\otimes M)\rtimes_{r}\Gamma$. Put $A:=(B\otimes M)\rtimes \Gamma \otimes C_{\lambda}^{\ast}(\Gamma)$ and define a u.c.p. map by
\begin{eqnarray*}
\Psi:=\pi \circ \Psi_{B\otimes M} \colon A \longrightarrow (B\otimes M)\rtimes_{r}\Gamma \longrightarrow \mathbb{B}(K),
\end{eqnarray*}
where $\Psi_{B\otimes M}$ is defined in Lemma $\ref{lemma B}$. 
By Steinspring's dilation theorem, we can decompose $\Psi=(V^{\ast}\hspace{-0.2em}\cdot V)\circ \psi$, where $V \colon K \longrightarrow \widetilde{K}$ is an isometry, $\widetilde{K}$ is a Hilbert space, and $\psi \colon A \longrightarrow \mathbb{B}(\widetilde{K})$ is a $\ast$-homomorphism. We construct $\widetilde{K}$ by a usual manner, that is, $\widetilde{K}$ is obtained from $A\odot K$ and $\psi$ is given by the left multiplication. It is easy to see that $\psi$ is normal on $M$ since there is an equation 
\begin{equation*}
\langle\psi(x)(b\otimes\xi)\mid c\otimes\eta\rangle_{H}=\langle\Psi(c^{\ast}xb)\xi\mid\eta\rangle, \quad (x\in M\ b,c\in A\ \xi,\eta\in H).
\end{equation*}
Now, apply Lemma $\ref{lemma A}$ to $\psi$ and take the induced bounded $\ast$-homomorphism
\begin{equation*}
\widetilde{\psi} \colon (B^{\ast\ast}\otimes M)\rtimes \Gamma \otimes C_{\lambda}^{\ast}(\Gamma) \longrightarrow \mathbb{B}(\widetilde{H}).
\end{equation*}
Then, the composite map
\begin{equation*}
B^{\ast\ast}\otimes(M\rtimes_{r}\Gamma) \xrightarrow{\Phi_{B^{\ast\ast}\otimes M}} (B^{\ast\ast}\otimes M)\rtimes \Gamma \otimes C_{\lambda}^{\ast}(\Gamma) \xrightarrow{\ \quad\widetilde{\psi}\quad\ } \mathbb{B}(\widetilde{H}) \xrightarrow{ \hspace{0.8em}V^{\ast}\hspace{-0.2em}\cdot V\hspace{0.8em}} \mathbb{B}(H)
\end{equation*}
is our desired one, where $\Phi_{B^{\ast\ast}\otimes M}$ is as in Lemma $\ref{lemma B}$. To verify this, use the identity $\Psi_{B\otimes M}\circ\Phi_{B^{\ast\ast}\otimes M}=\mathrm{id}$ on $B\otimes M$.

For a countable $\Gamma$, the last statement is trivial. To see the general case, let us take an increasing net $(\Gamma_i)_i$ of countable subgroups in $\Gamma$ such that $\Gamma=\bigcup_i \Gamma_i$, and construct corresponding conditional expectations $E_i \colon M\mathbin{\bar{\rtimes}}\Gamma\rightarrow M\mathbin{\bar{\rtimes}}\Gamma_i$
which is obtained by a compression of $1\otimes e_i$, where $e_i$ is the orthogonal projection onto $\ell^2(\Gamma_i)\subset\ell^2(\Gamma)$. Then $E_i$ converges to $\mathrm{id}_{M\mathbin{\bar{\rtimes}}\Gamma}$ in the point $\sigma$-weak topology and so $M\mathbin{\bar{\rtimes}}\Gamma$ is weakly exact.
\end{proof}

\begin{Cor}
Keep the notation above and assume further that $M$ has a faithful normal state $\omega$ which is preserved by the $\Gamma$-action. Then $M\mathbin{\bar{\rtimes}} \Gamma$ is weakly exact.
\end{Cor}
\begin{proof}
By the same trick as the proof above, we may assume $\Gamma$ is countable. We use a similar argument to the proof of Lemma $\ref{Tak lem}$. 

Take any sequence $(a_{n})_{n}$ in $M$ and define $M_{0}$ by the unital von Neumann subalgebra of $M$ generated by $\sigma_{t_{1}}^{\omega}\circ \alpha_{g_{1}}\circ \sigma_{t_{2}}^{\omega} \cdots \sigma_{t_{k}}^{\omega}\circ \alpha_{g_{k}}(a_{n})$, where $n,k\in \mathbb{N},\ t_{i}\in\mathbb{Q},\ g_{i}\in\Gamma$, $\sigma^{\omega}$ is the modular automorphism group of $\omega$, and $\alpha$ is the $\Gamma$-action. 
Then, $M_{0}$ is $\Gamma$-invariant and $\sigma^{\omega}$-invariant and its predual is separable. 
By Takesaki's conditional expectation theorem, there exists the $\omega$-preserving normal conditional expectation $E_{0 }$ from $M$ onto $M_{0}$ and it is naturally $\Gamma$-equivariant, since $\omega$ preserves the $\Gamma$-action. Hence, there is the extended normal conditional expectation from $M\mathbin{\bar{\rtimes}} \Gamma$ onto $M_{0}\mathbin{\bar{\rtimes}} \Gamma$. Now, since $M_{0}\mathbin{\bar{\rtimes}} \Gamma$ is weakly exact, we are done.
\end{proof}

\subsection{\bf Weakly semiexactness and free products}

In this subsection, we investigate weak exactness of (amalgamated) free products. For both groups and $C^{\ast}$-algebras, it is well-known that exactness passes to amalgamated free products $\cite{Dy ex}$. On the other hand, however, nuclearity does not so (for example, $C_{\lambda}^{\ast}(\mathbb{Z})\ast C_{\lambda}^{\ast}(\mathbb{Z})=C_{\lambda}^{\ast}(\mathbb{F}_{2})$ is not nuclear). 
The difference essentially comes from the fact that exactness passes to subalgebras (subspaces) but nuclearity does not. Since weak exactness does not pass to subalgebras, it is not so simple to prove that weak exactness is preserved under free products (at least, the author could not prove it in full generality). For this reason, we will give only some partial answers.

We first consider a case in which we can use exactness of $C^{\ast}$-algebras. Recall a von Neumann algebra is said to be $\it semiexact$ if it contains a $\sigma$-weakly dense exact $C^{\ast}$-algebra. It is easy to see that the free product of semiexact von Neumann algebras is still semiexact, since exactness is preserved under free products. We generalize this semiexactness as follows.

\begin{Def}\label{wk s-ex}\upshape
A von Neumann algebra $M$ is said to be $\it weakly$ $\it semiexact$ if there exist a semiexact von Neumann algebra $N$ and an (possibly non-unital) inclusion $M\subset N$ with a normal conditional expectation from $N$ onto $M$.
\end{Def}

For example, $L\Gamma$ is semiexact for any exact group $\Gamma$ and $pL\Gamma p$ is weakly semiexact for any projection $p\in L\Gamma$. Since weak semiexactness passes to tensor products, direct sum and free products (we will soon prove it), much many weakly exact von Neumann algebras are weakly semiexact. 

To prove weak semiexactness for free products, we observe that we may replace a conditional expectation by non-degenerate and unital one in the definition above.

\begin{Lem}
For any weakly semiexact von Neumann algebra $M$ and any faithful normal state $\omega$ on $M$, we can find a semiexact von Neumann algebra $N$ such that there exist an unital inclusion $M\subset N$ and a normal conditional expectation from $N$ onto $M$ with $\omega\circ E$ non-degenerate on $N$.
\end{Lem}
\begin{proof}
Let $N$ and $E$ be as in Definition $\ref{wk s-ex}$ and put $\tilde{\omega}:=\omega\circ E$. We regard $L^{2}(M,\omega)$ as a closed subspace of $L^{2}(N,\tilde{\omega})$ and denote the corresponding orthogonal projection by $e$. Note that $e$ is the extension of $E$ and satisfies $exe=E(x)e$ for any $x\in N$. 
Let $\pi$ be the GNS-representation of $\tilde{\omega}$ and $z\in N$ be the central projection with $\ker\pi=(1-z)N$. 
Then, it is easy to see that $L^{2}(N,\tilde{\omega})=L^{2}(zN,\tilde{\omega})$, $E(z)=1$, and the compression by $z$ on $M$ is injective. 
Hence, the GNS representation of $\tilde{\omega}$ on $zN$ is injective and we can identify $M\simeq zM$ with $\omega(\cdot)=\tilde{\omega}(z\cdot)$ on $M$. 
Now, the inclusion $zM\subset zN$, $\tilde{E} \colon zN\ni zx \mapsto zE(zx)\in zM$, and $\tilde{\omega}$ satisfy our desired conditions. Since $N \rightarrow zN$ is a normal $\ast$-homomorphism, $zN$ is still semiexact.
\end{proof}

Now, it is easy to prove that weak semiexactness passes to free products.
\begin{Pro}
Let $M_{i}$ $(i\in I)$ be von Neumann algebras with faithful normal states $\omega_{i}$ on $M_{i}$. If each $M_{i}$ is weakly semiexact, then the free product $\bar{\ast}_{i\in I} (M_{i},\omega_{i})$ is weakly semiexact.
\end{Pro}
\begin{proof}
By the previous lemma, we can find semiexact von Neumann algebras $N_{i}\supset M_{i}$ and non-degenerate normal conditional expectations $E_{i}$ for all $i\in I$. Then, there exists the normal conditional expectation $\ast_{i\in I} E_{i}$ from $\bar{\ast}_{i\in I} (N_{i},\omega_{i}\circ E_{i})$ onto $\bar{\ast}_{i\in I} (M_{i},\omega_{i})$, and so the proof is completed since $\bar{\ast}_{i\in I} (N_{i},\omega_{i}\circ E_{i})$ is semiexact, 
\end{proof}


Next, we go on the study with the Toeplitz--Pimsner algebras which has some relationship to free products. Indeed, to prove nuclearity on free products in some situations, the nuclearity of the Toeplitz--Pimsner algebras is sometimes useful. We prove weak exactness of free products in a special situation with a similar manner.

We first prove weak exactness (and semidiscreteness) of the Toeplitz--Pimsner algebras. We use the same notation as in Subsection $\ref{free}$.
\begin{Pro}
Let $M$ be a von Neumann algebra with a faithful normal state $\omega$ and $H$ be a $M$-$M$ correspondence on which $M$ acts normally and faithfully by left.
\begin{itemize}
	\item[$(1)$]If $M$ is semidiscrete, then ${\cal T}(H)\hookrightarrow {\cal T}(H)''$ is weakly nuclear. In particular, ${\cal T}(H)''$ is semidiscrete.
	\item[$(2)$]If $M$ is weakly exact, then ${\cal T}(H)$ is weakly exact in $\mathbb{B}(L^{2}({\cal T}(H),\omega\circ E))$. In particular, ${\cal T}(H)''$ is weakly exact if $M$ has separable predual and $H$ is countably generated on $M$.
\end{itemize}
\end{Pro}
\begin{proof}
(1) For simplicity, we write $\widetilde{H}:=L^{2}({\cal T}(H),\omega\circ E)$. Since $M$ is semidiscrete, the following map is min-continuous:
\begin{equation*}
\pi \colon M\otimes {\cal T}(H)' \longrightarrow \mathbb{B}(\widetilde{H})\ ;\ x\otimes y \longmapsto xy.
\end{equation*}
Then, consider a well-defined map
\begin{equation*}
\tau \colon H\odot {\cal T}(H)' \longrightarrow \mathbb{B}(\widetilde{H})\ ;\ \xi \otimes x \longmapsto \pi(T_{\xi})x.
\end{equation*}
and verify relations
\begin{equation*}
\tau(a\xi b)=\pi(a)\tau(\xi)\pi(b),\quad
\tau(\xi)^{\ast}\tau(\eta)=\pi(\langle\xi,\eta\rangle),
\end{equation*}
for any $a,b\in M\otimes {\cal T}(H)'$ and $\xi,\eta\in H\odot {\cal T}(H)'$. These relations give well-definedness and boundedness of $\tau$ on $H\otimes_{{\rm ext}} {\cal T}(H)'$. 
By universality of the Toeplitz--Pimsner algebras, there exists a $\ast$-homomorphism from ${\cal T}(H\otimes_{{\rm ext}}{\cal T}(H)')$ into $\mathbb{B}(\widetilde{H})$. Since ${\cal T}(H\otimes_{{\rm ext}}{\cal T}(H)')\simeq{\cal T}(H)\otimes{\cal T}(H)'$, we have min-boundedness of ${\cal T}(H)\otimes{\cal T}(H)' \rightarrow \mathbb{B}(\widetilde{H})$ which is equivalent to weak nuclearity of ${\cal T}(H)\hookrightarrow {\cal T}(H)''$. One must verify that this is really the map we need, but it is routine.

(2) We will use a similar way. Let $\pi,\tilde{\pi},B,J,$ and $\mathbb{B}(K)$ be as in Definition $\ref{def wk ex}$. 
Since $M$ is weakly exact, $\tilde{\pi}$ is bounded on $M\otimes(B/J)$. Define $\tau$ by
\begin{equation*}
\tau \colon H\otimes_{{\rm ext}} (B/J) \longrightarrow \mathbb{B}(K)\ ;\ \xi \otimes b \longmapsto \pi_{M}(T_{\xi})\pi_{(B/J)}(b),
\end{equation*}
where $\pi_{M}$ and $\pi_{(B/J)}$ are obtained from the decomposition $\pi=\pi_{M}\times\pi_{(B/J)}$ on $M\otimes(B/J)$. 
Then, by universality and ${\cal T}(H\otimes_{{\rm ext}}(B/J))\simeq{\cal T}(H)\otimes(B/J)$, we have boundedness of $\tilde{\pi}$.
\end{proof}

To state the following corollary, we define a $C^{\ast}$-correspondence associated with a u.c.p. map. Let $A$ be a unital $C^{\ast}$-algebra and $\phi$ be a u.c.p. map on $A$. 
We define an $A$-valued semi inner product on $A\odot A$ by $\langle a\otimes b,c\otimes d\rangle:=b^{\ast}\phi(a^{\ast}c)d$ and construct the Hilbert $A$-module $H_{A}^{\phi}$ from $A\odot A$ by separation and completion. 
Define a left $A$-action on $H_{A}^{\phi}$ by $a(b\otimes c):=ab\otimes c$ so that $H_{A}^{\phi}$ is an $A$-$A$ correspondence. Now, we can define ${\cal T}(H_{A}^{\phi})$ for faithful $\phi$. Note that, in the case where $A$ is a von Neumann algebra, the left $A$-action is faithful and normal if $\phi$ is so.
\begin{Cor}
Let $M$ be a von Neumann algebra and $D\subset M$ be a unital von Neumann subalgebra with a faithful normal conditional expectation $E_{D}$. Put ${\cal M}:=(M,E_{D})\mathbin{\bar{\ast}}_{D}(\mathbb{B}(H)\mathbin{\bar{\otimes}}D,\omega \otimes \mathrm{id})$, where $H$ is a separable Hilbert space (possibly finite dimensional) and $\omega$ is a vector state on $\mathbb{B}(H)$.
\begin{itemize}
	\item[$(1)$] If $M$ is semidiscrete$, then \ {\cal M}$ is semidiscrete.
	\item[$(2)$] If $M$ is weakly exact and has separable predual, then ${\cal M}$ is weakly exact.
\end{itemize}
\end{Cor}
\begin{proof}
By $\cite[\rm Exercise\ 4.8.1]{BO}$, we have a natural isomorphism $({\cal T}(H_{M}^{E_{D}}),E_{D}\circ E)\simeq(M,E_{D})\ast_{D}({\cal T}(\mathbb{C})\otimes D,\omega\otimes\mathrm{id}_{D})$, where $\omega$ is the canonical conditional expectation on ${\cal T}(\mathbb{C})$ (i.e. the vacuum state on ${\cal T}(\mathbb{C})$). 
The $\sigma$-weak closure of the left hand side is weakly exact (respectively, semidiscrete) since it satisfies the assumption in the previous proposition. Hence the $\sigma$-weak closure of the right hand side which is isomorphic to $\cal M$ with the infinite $H$ is weakly exact (respectively, semidiscrete), since the above isomorphism preserves the canonical conditional expectations.

Thus, we proved this corollary for the infinite $H$. Finally, take any $n$-dim projection $p$ in $\mathbb{B}(H)$ whose range contains $(1_{{\cal T}(\mathbb{C})})^{\wedge}$. Then the compression by $p$ is a normal conditional expectation from $\mathbb{B}(H)$ onto $p\mathbb{B}(H) p\simeq \mathbb{M}_{n}$ which preserves $\omega$. 
Hence, we have a normal conditional expectation from $\cal M$ with the infinite $H$ onto $\cal M$ with an $n$-dim $H$. This completes the proof.
\end{proof}

\section{\bf Application to Condition (AO)}\label{ap ao}

In $\cite{solid 1}$, $\cite{solid 2}$, and $\cite{solid 3}$, N. Ozawa proved celebrated theorems for solid von Neumann algebras which gave a new insight on group von Neumann algebras ($\cite{solid 2}$ was with S. Popa). Later, in his book $\cite{BO}$, he generalized his results with bi-exactness of groups. In this section, we generalize this theorem to crossed products with bi-exact groups and we then give some new examples of prime factors. In a part of the generalization, we will use our weak exactness.

\subsection{\bf Condition (AO)}\label{co ao}

We first recall condition (AO), (semi)solidity and primeness.
We say a von Neumann algebra $M\subset \mathbb{B}(H)$ satisfies $\it condition$ (AO) if there exist unital $\sigma$-weakly dense $C^{\ast}$-subalgebras $B\subset M$ and $C\subset M'$ satisfying $B$ is locally reflexive and the following map is min-continuous:
\begin{equation*}
\nu \colon B\otimes C\longrightarrow \mathbb{B}(H)/\mathbb{K}(H)\ ;\ b\otimes c\longmapsto bc+\mathbb{K}(H).
\end{equation*}
Here $\it locally$ $\it reflexive$ ($\it property$ $C''$) means that the inclusion $B^{\ast\ast}\odot A\hookrightarrow (B\otimes A)^{\ast\ast}$ is isometric on $B^{\ast\ast}\otimes A$ for any $C^{\ast}$-algebra $A$.
We say a finite von Neumann algebra $M$ is $\it solid$ (respectively, $\it semisolid$) if for any diffuse (respectively, type II) unital von Neumann subalgebra $N\subset M$, the relative commutant $N'\cap M$ is injective. We say a von Neumann algebra $M$ is $\it prime$ if for any tensor decomposition $M=M_{1}\mathbin{\bar{\otimes}}M_{2}$, either one of $M_{i}$ ($i=1,2$) is of type I. 
It is easy to see that, for any non-injective finite von Neumann algebras, solidity implies semisolidity and semisolidity implies primeness.

Ozawa proved that a finite von Neumann algebra is solid if it satisfies condition (AO) $\cite{solid 1}$. In particular, he proved that $L\Gamma$ is solid if $\Gamma$ is bi-exact (see Definition $\ref{bi-ex def}$).
 In the proof, local reflexivity is used to show the following statement:
\begin{itemize}
\item for normal u.c.p. maps $\phi\colon M\rightarrow \mathbb{B}(K)$ and $\psi \colon M' \rightarrow \mathbb{B}(K)$ which have commuting ranges, if the u.c.p. map $\phi\times\psi \colon M\odot C\ni b\otimes c \mapsto \phi(b)\psi(c)\in\mathbb{B}(K)$ is min-bounded on $B\otimes C$, then $\phi\times\psi$ is min-bounded on $M\otimes C$.
\end{itemize}
If $B\subset M$ and $C\subset M'$ satisfy this condition, we do not need local reflexivity of $B$. Hence, in stead of this, we can use weak exactness of $C\subset M'$. 
Indeed, if $C$ is weakly exact in $M'$ and $\phi,\psi$ are as above, then it is not so difficult to see that $(\phi\mid_{B})^{\ast\ast}\times\psi \colon B^{\ast\ast}\odot C\rightarrow \mathbb{B}(K)$ is min-bounded (use a similar manner to as in the proof of Theorem $\ref{we ex thm}$ (1) $\rm(v)\Rightarrow(i)$). 
Then, by the normality of $\phi$ on $M$ and since $B$ is unital, the restriction of $(\phi\mid_{B})^{\ast\ast}\times\psi$ on $B^{\ast\ast}z_{M}\otimes C\simeq M\otimes C$ coincides with the extension of $\phi\times\psi$ on $M\otimes C$, and so $\phi\times\psi$ is bounded on $M\otimes C$.

Summary, a key lemma of Ozawa's theorem ($\cite[\rm Lemma\ 5]{solid 1}$) is translated as follows.
\begin{Lem}\label{wk ex ao}
Let $M\subset \mathbb{B}(H)$ be a von Neumann algebra, $p\in M$ be a projection, and $N\subset pMp$ be a unital von Neumann subalgebra with a normal conditional expectation $E_{N}$ from $pMp$ onto $N$. Let $B\subset M$ and $C\subset M'$ be unital $\sigma$-weakly dense $C^{\ast}$-subalgbras with $C$ weakly exact in $M'$. Assume that the u.c.p. map
\begin{equation*}
\Phi_{N}\colon B\odot C\longrightarrow \mathbb{B}(pH)\ ;\ b\otimes c\longmapsto E_{N}(pbp)cp 
\end{equation*}
is min-continuous on $B\otimes C$. Then $N$ is injective.
\end{Lem}

\subsection{\bf Popa's intertwining techniques}

In $\cite{Po 1}$, S. Popa introduced a powerful tool to study $\cal HT$ factors and these fundamental groups, and he improved this technique in $\cite{Po 2}$. 
It gives a useful sufficient condition on unitary conjugacy of Cartan subalgebras of a finite von Neumann algebra, but Ozawa and Popa used this tool in a different context to study solidness. In this subsection, we recall this technique and generalize it for non-finite von Neumann algebras.

\begin{Def}\label{em def}\upshape
Let $M$ be a $\sigma$-finite von Neumann algebra, $A\subset M$ and $B\subset M$ be (possibly non-unital) von Neumann subalgebras, $E_{B}$ be a faithful normal conditional expectation from $1_{B}M1_{B}$ onto $B$. Assume $B$ is finite. We say $A$ $\it embeds$ $\it in$ $B$ $\it inside$ $M$ and denote by $A\preceq_{M}B$ if there exist non-zero projections $e\in A$ and $f\in B$, a unital normal $\ast$-homomorphism $\theta \colon eAe \rightarrow fBf$, and a partial isometry $v\in M$ such that
\begin{itemize}
\item the central support $z_{A}(e)$ is finite in $A$,
\item $vv^{\ast}\leq e$ and $v^{\ast}v\leq f$,
\item $v\theta(x)=xv$ for any $x\in eAe$.
\end{itemize}
\end{Def}

We often call above $v$ $\it intertwiner$. It is obvious that $A\preceq_{M}B$ holds if we have $pAp\preceq_{M}B$ for a projection $p\in A$. If $B$ is abelian and $A$ has no type I direct summand, then the existence of $\theta$ above says $A\not\preceq_{M}B$. The same thing happens for an injective $B$ and a non-injective $A$ satisfying $eAe$ is non-injective for any projection $e\in A$. We note that the condition on $A$ is equivalen to that $A$ has no injective direct summand, since 
\begin{equation*}
eAe \textrm{ \ injective} \Longleftrightarrow eA'e\simeq z(e)A' \textrm{ \ injective} \Longleftrightarrow z(e)A \textrm{ \ injective},
\end{equation*}
where $z(e)$ is the central support of $e$ $\cite[\rm Theorem\ XV.3.2]{Tak 3}$.
We will use these observation later.

Now, we recall Popa's intertwining techniques by bimodules in finite and semifinite settings. We mention that the finite case is generalized by Ueda and Houdayer-Vaes very recently.
\begin{Thm}[{\cite{Po 1}\cite{Po 2}\cite{Ueda}\cite{HV}}]\label{em thm}
Let $M,A,B$, and $E_{B}$ be as in Definition $\ref{em def}$, and let $\tau$ be a faithful normal trace state on $B$. Then the following conditions are equivalent.
\begin{itemize}
	\item[$(1)$]The algebra $A$ embeds in $B$ inside $M$.
	\item[$(2)$]There exist a projection $p\in \mathbb{M}_{n}\otimes B$, a unital normal $\ast$-homomorphism $\theta \colon A \rightarrow p(\mathbb{M}_{n}\otimes B)p$ with $n\geq 1$, and a partial isometry $w\in \mathbb{M}_{n}\otimes M$ such that 
\begin{itemize}
\item[$\bullet$] $ww^{\ast}\leq e_{1,1}\otimes1$ and $w^{\ast}w\leq p$,
\item[$\bullet$] $w\theta(x)=(e_{1,1}\otimes x)w$ for any $x\in A$.
\end{itemize}
	\item[$(3)$]There exists no net $(w_{\lambda})_{\lambda}$ of unitaries of $A$ such that  $E_{B}(b^{\ast}w_{\lambda}a)\rightarrow 0$ strongly (or equivalently,  $\displaystyle\lim_{\lambda}\|E_{B}(b^{\ast}w_{\lambda}a)\|_{2,\tau}=0$) for any $a,b\in 1_{A}M1_{B}$.
	\item[$(4)$]There exists a non-zero $A$-$B$-submodule $H$ of $1_{A}L^{2}(M)1_{B}$ with $\dim_{B}H<\infty$, where the dimension on $B$ is with respect to $\tau$.
\end{itemize}
In the case where $M$ has separable predual, we can take a sequence instead of a net in statement $(3)$.
\end{Thm}
\begin{Thm}[{\cite{CH}\cite{HR}}]\label{em thm2}
Let $M$ be a semifinite von Neumann algebra  with separable predual and with a faithful normal semifinite tracial weight $\mathrm{Tr}$, and $B\subset M$ be a unital von Neumann subalgebra with $\mathrm{Tr}_{B}:=\mathrm{Tr}\mid_{B}$ semifinite. Denote by $E_{B}$  the unique $\mathrm{Tr}$-preserving conditional expectation from $M$ onto $B$. For a (possibly non-unital) von Neumann subalgebra $A$, the following conditions are equivalent.
\begin{itemize}
	\item[$(1)$]There exists a non-zero projection $q\in B$ with $\mathrm{Tr}_{B}(q)<\infty$ such that $A\preceq_{pMp}qBq$, where $p:=1_{A}\vee q$.
	\item[$(2)$]There exists no sequence $(w_n)_n$ of unitaries of $A$ such that $E_{B}(b^{\ast}w_n a)\rightarrow 0$ strongly for any $a,b\in 1_{A}M1_{B}$.
	\item[$(3)$]There exists no sequence $(w_{n})_{n}$ of unitaries of $A$ such that $\displaystyle\lim_{n}\|E_{B}(b^{\ast}w_n a)\|_{2,\mathrm{Tr}_{B}}=0$ for any $a,b\in 1_{A}M1_{B}$.
\end{itemize}
\end{Thm}
We use the same symbol $A\preceq_{M}B$ if Theorem $\ref{em thm2}$ holds for semifinite $B\subset M$. Note that these definition does not depend on the choice of traces.

Ozawa used only the implication $(4)\Rightarrow(1)$ in Theorem $\ref{em thm}$ to prove his theorem on bi-exactness. So we will use only this implication for the main theorem, since our proof is very similar to his original proof. (For this reason, we used only Asher's observation $\cite[\rm Proposition\ 3.2]{As}$ in the first version of this paper.)

Finally, we recall a useful technique which allows us to see only abelian subalgebras $\cite[\rm Corollary\ F.14]{BO}$. The statement here is slightly general, but the proof is almost identical.

\begin{Pro}\label{cor F}
Let $M$ be a semifinite von Neumann algebra with a faithful semifinal normal weight $\omega$ and with separable predual. Let $A$ and $B_{n}$ $(n\in\mathbb{N})$ be (possibly non-unital) von Neumann subalgebras in $M$ such that $\omega(1_{A})<\infty$,  $1_{B_{n}}=1_{B_{m}}$ $(n,m\in \mathbb{N})$, and $\omega\mid_{B_{n}}:=\mathrm{Tr}_{n}$ $(n\in\mathbb{N})$ is semifinite and tracial.  
Assume one of the following conditions:
\begin{itemize}
\item (semifinite setting) each $B_{n}$ is a unital subalgebra of $M$ and $\omega$ is tracial on $M$;
\item (finite setting) $\omega(1_{M})=1$.
\end{itemize}
If $A$ does not embed in $B_{n}$ inside $M$ for every $n\in\mathbb{N}$, then there exists a unital abelian von Neumann subalgebra $A_{0}\subset A$ such that $A_{0}$ does not embed in $B_{n}$ inside $M$ for every $n\in\mathbb{N}$.
\end{Pro}
\begin{proof}
We may assume each $B_{k}$ appears infinitely often in $(B_{n})_{n}$. Let $(x_{n})_{n}$ be a $\sigma$-$\ast$strongly dense sequence in the closed unit ball of $M$. 
Then, by the completely same manner as in $\cite[\rm Corollary\ F.14]{BO}$, we can construct an increasing sequence $A_{n}\subset A$ of finite dimensional abelian unital von Neumann subalgebras and unitaries $w_{n}\in A_{n}$ satisfying $\|E_{B_{n}}(x_{j}^{\ast}w_{n}x_{i})\|_{2,\mathrm{Tr}_{n}}\leq n^{-1}$ for any $n\in\mathbb{N}$ and $i,j\leq n$.

We assume the first condition (and write $\mathrm{Tr}:=\omega$). Take any element $a$ in the closed unit ball of $M$ and any $\epsilon>0$. Find a subsequence $(x_{i})$ such that $x_{i}\rightarrow a$ in the $\sigma$-$\ast$strong topology. Since $(x_{i}-a)(x_{i}-a)^{\ast}$ converges zero $\sigma$-weakly and $\mathrm{Tr}(1_{A}\cdot 1_{A})$ is bounded normal, we have 
\begin{equation*}
\|1_{A}x_{i}-1_{A}a\|_{2,\mathrm{Tr}}^{2}=\mathrm{Tr}(1_{A}(x_{i}-a)(x_{i}-a)^{\ast}1_{A})\longrightarrow 0.
\end{equation*}
Hence, we can find $x_{i}$ satisfying $\|1_{A}x_{i}-1_{A}a\|_{2,\mathrm{Tr}}<\epsilon$.
 Now, we have the following inequality for any $j$ and large $n$:
\begin{eqnarray*}
\|E_{B_{n}}(x_{j}^{\ast}w_{n}a)\|_{2,\mathrm{Tr}}
&\leq&\|E_{B_{n}}(x_{j}^{\ast}w_{n}x_{i})\|_{2,\mathrm{Tr}}+\|E_{B_{n}}(x_{j}^{\ast}w_{n}(x_{i}-a))\|_{2,\mathrm{Tr}}\\
&\leq& n^{-1}+\|x_{j}^{\ast}w_{n}(x_{i}-a)\|_{2,\mathrm{Tr}}\\
&\leq& n^{-1}+\|x_{j}^{\ast}w_{n}\|\|1_{A}x_{i}-1_{A}a\|_{2,\mathrm{Tr}}\\
&<& n^{-1}+\ \epsilon.
\end{eqnarray*}
This means $\|E_{B_{n}}(x_{j}^{\ast}w_{n}a)\|_{2,\mathrm{Tr}}\leq n^{-1}$, and we can also replace $x_{j}$ by any $b\in 1_{A}M$. This completes the first case.

Now, we treat the second case. Take any $a\in M$ and $\epsilon>0$ and find $x_{i}\in M$ such that $\|x_{i}-a\|_{2,\omega}<\epsilon$. Then, for any $j$ and large $n$, we have 
\begin{eqnarray*}
\|E_{B_{n}}(x_{j}^{\ast}w_{n}a)\|_{2,\mathrm{Tr}_{n}}
&=&\|E_{B_{n}}(x_{j}^{\ast}w_{n}x_{i})\|_{2,\mathrm{Tr}_{n}}+\|E_{B_{n}}(x_{j}^{\ast}w_{n}(x_{i}-a))\|_{2,\mathrm{Tr}_{n}}\\
&\leq& n^{-1}+\|x_{j}^{\ast}w_{n}(x_{i}-a)\|_{2,\omega}\\
&\leq& n^{-1}+\|x_{j}^{\ast}w_{n}\|\|x_{i}-a\|_{2,\omega}\\
&<& n^{-1}+\ \epsilon.
\end{eqnarray*}
To replace $x_{j}$ by any element $b\in M$, use the equality $\|E_{B_{n}}(y)\|_{2,\mathrm{Tr}_{n}}=\|E_{B_{n}}(y^{\ast})\|_{2,\mathrm{Tr}_{n}}$ which comes from the trace condition on $B_{n}$.
\end{proof}

\subsection{\bf Bi-exactness and main theorem}

After showing solidity of some group von Neumann algebras, Ozawa and Popa generalized this results to tensor products, crossed products, and free products. Ozawa's theorem on bi-exactness includes most of these results. In this subsection, we recall bi-exactness on groups and prove the main theorem in this section which generalizes Ozawa's theorem.

Let $\Gamma$ be a discrete group and ${\cal G}$ be a family of subgroups of $\Gamma$. We say a subset $\Omega \subset \Gamma$ is $\it small$ $\it relative$ $\it to$ ${\cal G}$ if there exists $n\in\mathbb{N}$, $\Lambda_{i}\in {\cal G}$ and $s_{i}, t_{i}\in\Gamma$ $(i=1,2,\ldots,n)$ such that $\Omega\subset\bigcup_{i}s_{i}\Lambda t_{i}$.
We let $c_{0}(\Gamma;{\cal G})\subset \ell^{\infty}(\Gamma)$ be the $C^{\ast}$-subalgebra generated by functions whose support are small relative to ${\cal G}$. Note that $\Gamma$ has a left (respectively, right) action on $\ell^{\infty}(\Gamma)$/$c_{0}(\Gamma;{\cal G})$ induced from the left (respectively, right) translation. 

Now, we define bi-exactness of groups. The definition here is different from original one (see $\cite[\rm Definition\ 15.1.2]{BO}$), but they give the same condition $\cite[\rm Proposition\ 15.2.3]{BO}$.
\begin{Def}\label{bi-ex def}\upshape
For above $\Gamma$ and ${\cal G}$, we say $\Gamma$ is $\it bi$-$\it exact$ $\it relative$ $\it to$ ${\cal G}$ if the left-right action of $\Gamma\times\Gamma$ on $\ell^{\infty}(\Gamma)/c_{0}(\Gamma;{\cal G})$ is amenable. We simply say $\Gamma$ is $\it bi$-$\it exact$ (or $\Gamma$ is in $\it class$ $\cal S$) if ${\cal G}$ consists of only the trivial subgroup.
\end{Def}

Examples of bi-exact groups are all hyperbolic groups $\cite{HG}$ (in particular, all finitely generated free groups $\cite{AO}$) and many relatively hyperbolic groups $\cite{Oz rel hyp}$; lattices of simple rank one Lie groups $\cite{Sk}$; $\mathbb{Z}^2\rtimes \mathrm{SL}(2,\mathbb{Z})$ $\cite{Oz 09}$; wreath product of bi-exact groups on amenable groups $\cite{solid 3}$; subgroups of bi-exact groups. Moreover it is known that bi-exactness is ME invariant $\cite{Sa 09}$.

Next, we prepare a key lemma for our main theorem. Let $A$ be a semifinite von Neumann algebra with a faithful semifinite normal tracial weight Tr and $\Gamma$ be a discrete group acting on $A$. Denote this $\Gamma$-action by $\alpha$. 
Let $\pi,\lambda,\rho,J_{A},J$, and $v_{g}$ be as in Subsection $\ref{cros}$ and put $B:=A\rtimes_{r}\Gamma$ and $C:=JBJ$. We note that $C$ is a $\sigma$-weakly dense $C^{\ast}$-subalgebra of $B'$ and
\begin{eqnarray*}
C=JBJ&=&C^{\ast}\{J_{A}aJ_{A}\otimes 1,v_{s}\otimes \rho_{s}\mid a\in A,s\in\Gamma\}\\
&\simeq&C^*\{J_{A}aJ_{A}\otimes 1,v_{s}\otimes \lambda_{s}\mid a\in A,s\in\Gamma\}\\
&\simeq&(J_{A}AJ_{A})\rtimes_{r}\Gamma,
\end{eqnarray*}
where the second isomorphism is given by the unitary $\xi\otimes \delta_g \longmapsto \xi\otimes \delta_{g^{-1}}$ and the $\Gamma$-action on $J_{A}AJ_{A}$ is given by $J_{A}aJ_{A}\mapsto J_{A}\alpha_{g}(a)J_{A}$ for $a\in A$ and $g\in \Gamma$. 

Finally, assume $\Gamma$ is bi-exact relative to ${\cal G}$ and let $\mathbb{K}(A,\Gamma;{\cal G})$ be the hereditary $C^{\ast}$-subalgebra generated by $1\otimes c_{0}(\Gamma;{\cal G})$ (i.e. the norm closure of $(1\otimes c_{0}(\Gamma;{\cal G}))\mathbb{B}(L^{2}(A)\otimes\ell^{2}(\Gamma))(1\otimes c_{0}(\Gamma;{\cal G}))$). We set a closed ideal $I:=\mathbb{K}(A,\Gamma;{\cal G})\cap C^{\ast}(B,C) \subset C^{\ast}(B,C).$

Here we give a generalization of $\cite[\rm Proposition\ 4.2]{solid 3}$. The proof here is identical to that of $\cite[\rm Proposition\ 24]{Sako}$, but we include the proof for self-containment.
\begin{Pro}\label{bi-ex cross}
In the setting above,
\begin{eqnarray*}
\nu \colon B\odot C \longrightarrow \frac{C^{\ast}(B,C)}{I} \ ;\ 
b\otimes c \longmapsto bc + I
\end{eqnarray*}
is min-continuous. 
\end{Pro}
\begin{proof}
We first put $D:=C^{\ast}(\pi(A),J_{A}AJ_{A}\otimes1, 1\otimes \ell^{\infty}(\Gamma))\subset D_{1}:=C^{\ast}(B,C,1\otimes \ell^{\infty}(\Gamma))$ and $E:=(D+I)/I\ \subset E_{1}:=(D_{1}+I)/I\ \supset C^{\ast}(B,C)/I$. Then, define two $\Gamma$-actions on $E$ by
\begin{eqnarray*}
\left\{
\begin{array}{l}
\Gamma \longrightarrow \mathrm{Aut}(E)\ ;\ g \longmapsto \mathrm{Ad}(1\otimes \lambda_{g})\\
\Gamma \longrightarrow \mathrm{Aut}(E)\ ;\ g \longmapsto \mathrm{Ad}(v_{g}\otimes \rho_{g})
\end{array}
\right.
\end{eqnarray*}
and construct the $(\Gamma\times\Gamma)$-action on $E$ by them.

The inclusion $\ell^{\infty}(\Gamma)\simeq1\otimes\ell^{\infty}(\Gamma) \subset D+I$ induces the inclusion $\ell^{\infty}(\Gamma)/c_{0}(\Gamma;{\cal G})\simeq1\otimes(\ell^{\infty}(\Gamma)/c_{0}(\Gamma;{\cal G}))\subset E$ which is $(\Gamma\times\Gamma)$-equivariant. 
Since $1\otimes(\ell^{\infty}(\Gamma)/c_{0}(\Gamma;{\cal G}))$ is contained in the center of $E$ and the $(\Gamma\times\Gamma)$-action is amenable on $1\otimes(\ell^{\infty}(\Gamma)/c_{0}(\Gamma;{\cal G}))$ by assumption, this action is amenable on $E$. 
Hence, we have $E\rtimes_{r}(\Gamma\times\Gamma)=E\rtimes(\Gamma\times\Gamma)$. 
Since the inclusion map $E \hookrightarrow E_{1}$ and the unitary representation $\Gamma\times\Gamma \rightarrow E_{1}\ ;\ (s,t)\mapsto (1\otimes \lambda_{s})(v_{t}\otimes\rho_{t})$ are a covariant representation, we have a $\ast$-homomorphism
\begin{equation*}
\phi \colon E\rtimes_{r}(\Gamma\times\Gamma)=E\rtimes(\Gamma\times\Gamma) \longrightarrow E_{1}
\end{equation*}
by universality.

On the other hand, injectivity of $A$ implies min-boundedness of
\begin{equation*}
\pi(A)\otimes J\pi(A)J \ni a\otimes JbJ\longmapsto aJbJ \in D+I\subset\mathbb{B}(L^{2}(A)\otimes \ell^{2}(\Gamma)),
\end{equation*}
since this map is a restriction of the bounded map $\pi(A)\otimes \pi(A)'\rightarrow \mathbb{B}(L^{2}(A)\otimes \ell^{2}(\Gamma))$. Compose this map with the quotient map $D+I\ \rightarrow E$ and notice the resulting map is $(\Gamma\times\Gamma)$-equivariant, where the $(\Gamma\times\Gamma)$-action on $\pi(A)\otimes J\pi(A)J$ is given by $(s,t)\cdot(a\otimes JbJ):=\alpha_{s}(a)\otimes J\alpha_{t}(b)J$. Then, we have the extended $\ast$-homomorphism
\begin{equation*}
(\pi(A)\otimes J\pi(A)J)\rtimes_{r}(\Gamma\times\Gamma) \longrightarrow E\rtimes_{r}(\Gamma\times\Gamma).
\end{equation*}
Composing this map with $\phi$ and using an identification 
\begin{equation*}
(\pi(A)\otimes J\pi(A)J)\rtimes_{r}(\Gamma\times\Gamma)\simeq (\pi(A)\rtimes_{r}\Gamma)\otimes (J\pi(A)J\rtimes_{r}\Gamma)\simeq B\otimes C,
\end{equation*}
we are done.
\end{proof}

Before the main theorem, we consider the relationship between injective von Neumann algebras and these commutants. Let $N\subset\mathbb{B}(H)$ be an injective von Neumann algebra which has separable predual. Since it is AFD by Connes' theorem $\cite{Co}$, there exists an increasing sequence  $(N_{n})$ of finite dimensional unital $C^{\ast}$-subalgebras of $N$ with $\displaystyle N=(\bigcup_{n} N_{n})''$. Define a u.c.p. map on $\mathbb{B}(H)$ by
\begin{equation*}
\Psi_{N}(x):=\lim_{n\rightarrow \omega}\int_{{\cal U}(N_{n})}uxu^{\ast}du,
\end{equation*}
where $du$ is the normalized Haar measure on ${\cal U}(N_{n})$, $\omega$ is a fixed free ultra filter, and the limit is taken by the $\sigma$-weak topology. It is easily checked that $\Psi_{N}$ is a proper conditional expectation from $\mathbb{B}(H)$ onto $N'$. Here $\it proper$ means that $\Psi_{N}$ satisfies
\begin{equation*}
\Psi_{N}(x)\in \overline{\rm{co}}\{uxu^{\ast}\mid u\in {\cal U}(N)\}
\end{equation*}
for any $x\in\mathbb{B}(H)$, where the closure is taken by the $\sigma$-weak topology.

Now, we are in position to prove our main theorem which generalizes $\cite[\rm Theorem\ 15.1.5]{BO}$.
\begin{Thm}\label{bi ex thm}
Let $A$ be an injective semifinite von Neumann algebra with a faithful semifinite normal tracial weight $\mathrm{Tr}$, $\Gamma$ be a discrete group acting on $A$, and ${\cal G}$ be a family of subgroups of $\Gamma$. Put $M:=A\mathbin{\bar{\rtimes}}\Gamma$ and let $N\subset M$ be an injective finite (possibly non-unital) von Neumann subalgebra with separable predual. Assume the following conditions:
\begin{itemize}
\item $\Gamma$ is bi-exact relative to ${\cal G}$;
\item for any $\Lambda\in{\cal G}$, the induced $\Lambda$-action on $A$ is $\mathrm{Tr}$-preserving (so that $A\mathbin{\bar{\rtimes}}\Lambda$ is semifinite);
\item there exists a projection $p\in A$ such that $1_{N}\leq p$ and $\mathrm{Tr}(\alpha_{s}(p))<\infty$ for any $s\in\Gamma$;
\item there exists a faithful normal conditional expectation $E_{N}$ from $1_{N}M1_{N}$ onto $N$.
\end{itemize}
Then, either one of the following conditions holds.
\begin{itemize}
	\item[$(\rm{i})$]The relative commutant $N'\cap 1_{N}M1_{N}$ is injective.
	\item[$(\rm{ii})$]There exists $\Lambda \in {\cal G}$ and a projection $q\in A$ with $\mathrm{Tr}(q)<\infty$ such that $N\preceq_{eMe}q(A\mathbin{\bar{\rtimes}}\Lambda) q$, where $e:=1_{N}\vee q$.
\end{itemize}
\end{Thm}
\begin{proof}
For simplicity, we put $K:=L^{2}(A)\otimes\ell^{2}(\Gamma)$ and note that $K\simeq L^{2}(M,\omega)$ as a standard representation of $M$, where $\omega$ is the dual weight of $\rm Tr$.

Assume $N\preceq_{eMe}q(A\mathbin{\bar{\rtimes}}\Lambda) q$ for any $\Lambda \in {\cal G}$ and $q\in A$ with $\mathrm{Tr}(q)<\infty$.
By assumption, there exists a proper conditional expectation 
\begin{equation*}
\Psi_{N} \colon \mathbb{B}(1_{N}K) \longrightarrow N'\cap \mathbb{B}(1_{N}K),
\end{equation*}
and we extend $\Psi_{N}$ on $\mathbb{B}(K)$, composing with the compression map by $1_{N}$. Here we claim that $\mathbb{K}(A,\Gamma;{\cal G})\subset \ker \Psi_{N}$.

Since $\Psi_{N}$ is u.c.p, it suffice to show $1\otimes c_{0}(\Gamma, {\cal G})\subset \ker \Psi_{N}$, and so we will show $\Psi_{N}(1\otimes\chi_{s\Lambda t})=0$ for any $s,t\in \Gamma$ and $\Lambda\in{\cal G}$, where $\chi_{s\Lambda t}$ is the characteristic function on $s\Lambda t\subset \Gamma$. We may further assume $t=e$ (the unit of $\Gamma$), since
\begin{equation*}
\Psi_{N}(1\otimes\chi_{s\Lambda t})=\Psi_{N}(1\otimes(\rho_{t^{-1}}\chi_{s\Lambda}\rho_{t}))=(1\otimes\rho_{t^{-1}})\Psi_{N}(1\otimes\chi_{s\Lambda})(1\otimes\rho_{t}).
\end{equation*}
Take any finite projection $q\in A$ with $\mathrm{Tr}(q)<\infty$, and put $\tilde{q}:=JqJ\in M'\subset N'$, $\widetilde{N}:=N\tilde{q}$, and $\Psi_{\widetilde{N}}(x):=\tilde{q}\Psi_{N}(x)\tilde{q}$ for $x\in \mathbb{B}(K)$. Consider a natural right $M$ action on $K$ with $J$ and write $\tilde{q}1_N K=1_N Kq$. It is easy to see that  $\Psi_{\widetilde{N}}$ gives a proper conditional expectation from $\mathbb{B}(1_{N}Kq)$ onto $\widetilde{N}'\cap\mathbb{B}(1_{N}Kq)$. We will show $\Psi_{\widetilde{N}}(1\otimes\chi_{s\Lambda})=0$ and this says $\Psi_{N}(1\otimes\chi_{s\Lambda})=0$ since $q\in A$ is arbitrary.

The properness of $\Psi_{\widetilde{N}}$ says that $\Psi_{\widetilde{N}}(1\otimes\chi_{s\Lambda})$ is in the commutant of $Jq(A\mathbin{\bar{\rtimes}}\Lambda)qJ$, and so any spectral projection $r$ of $\Psi_{\widetilde{N}}(1\otimes\chi_{s\Lambda})$ is in $\widetilde{N}'\cap (Jq(A\mathbin{\bar{\rtimes}}\Lambda)qJ)' \cap \mathbb{B}(1_{N}Kq)$. Hence, the subspace $H:=rK\subset 1_{N}Kq\simeq 1_{N}L^{2}(M)q$ has a structure of $\widetilde{N}$-$q(A\mathbin{\bar{\rtimes}}\Lambda)q$-submodule of $L^{2}(M)$. 
By Theorem $\ref{em thm}$ and our assumption, $\dim_{qA\mathbin{\bar{\rtimes}}\Lambda q}H$ is zero or infinite.

Identifying $\Lambda\times(\Gamma/\Lambda)\ni(g,s\Lambda)\mapsto \sigma(s\Lambda)g\in\Gamma$ by a fixed section $\sigma :\Gamma/\Lambda\rightarrow\Gamma$ with $\sigma(\Lambda)=e$, 
we have a unitary $\ell^{2}(\Lambda)\otimes\ell^{2}(\Gamma/\Lambda)\ni \delta_{g}\otimes\delta_{s\Lambda}\mapsto \delta_{\sigma(s\Lambda)g}\in\ell^{2}(\Gamma)$. Notice that the right action of $\Lambda$ on $\ell^{2}(\Gamma)\simeq\ell^{2}(\Lambda)\otimes\ell^{2}(\Gamma/\Lambda)$ is of the form $\rho_{g}\otimes 1$ for $g\in\Lambda$. Hence we have 
\begin{equation*}
J(A\mathbin{\bar{\rtimes}}\Lambda)J\simeq W^{\ast}(a\otimes1\otimes1,v_{g}\otimes\rho_{g}\otimes1,(a\in A, g\in\Lambda))\subset\mathbb{B}(L^{2}(A)\otimes\ell^{2}(\Lambda)\otimes\ell^{2}(\Gamma/\Lambda)),
\end{equation*}
and this implies $(J(A\mathbin{\bar{\rtimes}}\Lambda )J)'\simeq(A\mathbin{\bar{\rtimes}}\Lambda)\mathbin{\bar{\otimes}}\mathbb{B}(\ell^{2}(\Gamma/\Lambda))$.
Then, there is the following identification: 
\begin{eqnarray*}
(Jq(A\mathbin{\bar{\rtimes}}\Lambda)qJ)'=(J(A\mathbin{\bar{\rtimes}}\Lambda) J)'\tilde{q}
\simeq (A\mathbin{\bar{\rtimes}}\Lambda)\mathbin{\bar{\otimes}}\mathbb{B}(\ell^{2}(\Gamma/\Lambda))\tilde{q}\simeq (A\mathbin{\bar{\rtimes}}\Lambda)\mathbin{\bar{\otimes}}\mathbb{B}(\ell^{2}(\Gamma/\Lambda))z(\tilde{q}),
\end{eqnarray*}
where $z(\tilde{q})$ is the central support of $\tilde{q}$ in $(A\mathbin{\bar{\rtimes}}\Lambda)\mathbin{\bar{\otimes}}\mathbb{B}(\ell^{2}(\Gamma/\Lambda))$.
 Hence we can find the projection $\tilde{r}\in (A\mathbin{\bar{\rtimes}}\Lambda)\mathbin{\bar{\otimes}}\mathbb{B}(\ell^{2}(\Gamma/\Lambda))z(\tilde{q})$ corresponding to $r$. Note that $\tilde{r}$ satisfies  $\tilde{r}\tilde{q}=r$ and $\tilde{r}\leq Cz(\tilde{q})\Psi_{N}(1\otimes\chi_{s\Lambda})$ for some $C>0$ (notice that $\Psi_{N}(1\otimes\chi_{s\Lambda})\in (A\mathbin{\bar{\rtimes}}\Lambda)\mathbin{\bar{\otimes}}\mathbb{B}(\ell^{2}(\Gamma/\Lambda))$). Then it holds $H=rKq=\tilde{r}Kq$ and $Kq$ has the following right $q(A\mathbin{\bar{\rtimes}}\Lambda) q$-module isometry:
\begin{equation*}
L^{2}(A\mathbin{\bar{\rtimes}}\Lambda)q\otimes\ell^{2}(\Gamma/\Lambda)\simeq
L^{2}(A)q\otimes\ell^{2}(\Lambda)\otimes\ell^{2}(\Gamma/\Lambda)\simeq
L^{2}(A)q\otimes\ell^{2}(\Gamma)=Kq.
\end{equation*}
Hence, by the observation in Subsection $\ref{dim}$, the dimension of $H$ on $q(A\mathbin{\bar{\rtimes}}\Lambda)q$ with $\omega(\cdot)/\omega(q)=:\tilde{\omega}$ is smaller than $\widetilde{\rm{Tr}}(\tilde{r})$, where $\widetilde{\rm{Tr}}:=\tilde{\omega}\otimes \mathrm{Tr_{\mathbb{B}(\ell^{2}(\Gamma/\Lambda))}}$ on  $(A\mathbin{\bar{\rtimes}}\Lambda)\mathbin{\bar{\otimes}} \mathbb{B}(\ell^{2}(\Gamma/\Lambda))\subset \mathbb{B}(L^{2}(A)\otimes \ell^{2}(\Lambda)\otimes\ell^{2}(\Gamma/\Lambda))$.
Now, the value of this is finite since
\begin{eqnarray*}
\widetilde{\rm{Tr}}(\tilde{r})
&\leq& C\cdot \widetilde{\rm{Tr}}(\Psi_{N}(1\otimes\chi_{s\Lambda}))\\
&\leq& C \cdot \widetilde{\rm{Tr}}(1_{N}(1\otimes\chi_{s\Lambda})1_{N})\\
&=& C \cdot \widetilde{\rm{Tr}}(1\otimes\chi_{s\Lambda})1_{N}(1\otimes\chi_{s\Lambda}))\\
&\leq& C \cdot \widetilde{\rm{Tr}}(1\otimes\chi_{s\Lambda})p(1\otimes\chi_{s\Lambda}))\\
&=& C \cdot \mathrm{Tr}(\alpha_{\sigma(s\Lambda)}(p))\\
&<&\infty,
\end{eqnarray*}
where we used properness of $\Psi_{N}$ and normality of $\widetilde{\rm{Tr}}$. Hence, $\widetilde{\rm{Tr}}(\tilde{r})=0$ and $\tilde{r}=r=0$. This means $\Psi_{\tilde{N}}(1\otimes\chi_{s\Lambda})=0$ and our claim is completed.

Now, combined with Proposition $\ref{bi-ex cross}$ and the claim above, we have a well-defined and min-bounded map
\begin{equation*}
\Phi_{N'\cap1_{N}M1_{N}}:=\Psi_{N}\circ \nu \colon B\otimes C \longrightarrow \frac{C^{\ast}(B,C)}{I} \longrightarrow N'\cap \mathbb{B}(1_{N}K),
\end{equation*}
where $\nu,B,C$, and $I$ are as in Proposition $\ref{bi-ex cross}$. We note that the map $\Phi_{N'\cap1_{N}M1_{N}}$ is same as one given in Lemma $\ref{wk ex ao}$ with the conditional expectation $\Psi_{N}$. To complete assumptions in this lemma, we need weak exactness of $C=J(A\rtimes_{r}\Gamma)J$ in $M'=JMJ$ and normality of $\Psi_{N}$ on $M$. 
We know $A\rtimes_{r}\Gamma$ is weakly exact in $M$ by Proposition $\ref{cro pro}$, and so it is not so difficult to see the weak exactness of $C$ in $M'$. For the normality of $\Psi_{N}$, we use the map $E_{N}$ in assumption.

Let $\tau$ be a faithful normal trace state on $N$ and put $\tilde{\tau}:=\tau\circ E_{N}$ on $1_{N}M1_{N}$. By properness, the restriction of $\Psi_{N}$ on $1_{N}M1_{N}$ gives the unique $\tilde{\tau}$-preserving conditional expectation onto $N'\cap1_{N}M1_{N}$, and so it is faithful and normal on $1_{N}M1_{N}$ by Takesaki's conditional expectation theorem.
 Thus, we can apply Lemma $\ref{wk ex ao}$ to $M,N$, and $\Psi_{N}$ so that we have injectivity of $N'\cap1_{N}M1_{N}$.
\end{proof}

\begin{Rem}\upshape\label{bi ex fin}
In this theorem, if we further assume the $\Gamma$-action is Tr-preserving, $\Gamma$ and $\cal G$ are countable, $A$ has separable predual (so that $M$ is semifinite with separable predual), then the same results holds for any (possibly non-unital) von Neumann subalgebra $N\subset pMp$, where $p$ is a projection in $A$ with $\mathrm{Tr}(p)<\infty$. 
Indeed, if $N\preceq_{eMe}q(A\mathbin{\bar{\rtimes}}\Lambda) q$ for any $\Lambda \in {\cal G}$ and $q\in A$ with $\mathrm{Tr}(q)<\infty$ (which exactly means $N\preceq_{M}A\mathbin{\bar{\rtimes}}\Lambda$ for any $\Lambda \in {\cal G}$ in semifinite setting), we can find an abelian unital von Neumann subalgebra $B\subset N$ satisfying the same condition as $N$, appealing Proposition $\ref{cor F}$. 
Then, injectivity of $B'\cap 1_{N}M1_{N}$ implies that of $N'\cap 1_{N}M1_{N}$ since they are finite.
\end{Rem}

\subsection{\bf Examples of prime factors}\label{ex prim}

As a corollary of our main theorem, we give some new examples of prime factors. In this subsection, for simplicity, we use the term $\it with$ $\it expectation$ for (non-unital) inclusion $N\subset M$ if there exists a faithful normal conditional expectation from $1_NM1_N$ onto $N$.

 The following corollary is a generalization of $\cite[\rm Theorem\ 4.6]{solid 3}$ and we use a similar way to as in the proof of $\cite[\rm Proposition\ 2.7]{GJ}$.

\begin{Cor}\label{ab prim}
Let $A$ be an abelian von Neumann algebra with separable predual and $\Gamma$ be a countable discrete group acting on $A$. Let $M$ be a (possibly non-unital) von Neumann subalgebra of $A\mathbin{\bar{\rtimes}}\Gamma$ with expectation and denote $\widetilde{M}:=1_{M}(A\mathbin{\bar{\rtimes}}\Gamma)1_{M}$. Assume $\Gamma$ is bi-exact.
\begin{itemize}
\item[$(1)$]If $M$ has no type $\rm I$ direct summand, then the relative commutant $M'\cap \widetilde{M}$ is injective.

\item[$(2)$]If $M$ is non-injective, then $M$ is prime.
\end{itemize}
\end{Cor}
\begin{proof}
(1) By $\cite[\rm Theorem\ 11.1]{HS}$, there exists a faithful normal state $\phi$ on $M$ such that $M_{\phi}$ is of type $\rm{II}_{1}$, where $M_{\phi}$ is the centralizer of $\phi$ defined by
\begin{equation*}
M_{\phi}:=\{x\in M\mid \sigma^{\phi}_{t}(x)=x\ {\rm for\ any}\ t\in\mathbb{R}\}.
\end{equation*}
Note that there exists the unique $\phi$-preserving conditional expectation from $M$ onto $M_{\phi}$. Since $M_{\phi}$ is of type $\rm II_{1}$, it has a copy of the hyperfinite $\rm II_{1}$ factor, say $N$, as a unital subalgebra. Now apply Theorem $\ref{bi ex thm}$ to $N$ with any faithful normal state on $A$ (which is automatically a trace), and we have injectivity of $N'\cap \widetilde{M}$, since  $N\preceq_{A\mathbin{\bar{\rtimes}}\Gamma}A$ does not happen ($N$ is of type $\rm{II}_{1}$ and $A$ is abelian).

Put $\tilde{\phi}:=\phi\circ E_{M}$ and observe $\sigma^{\tilde{\phi}}\mid_{M}=\sigma^{\phi}$ and $M_{\phi}\subset \widetilde{M}_{\tilde{\phi}}$ (this follows from the proof of Takesaki's conditional expectation theorem). Then, it holds that $\sigma^{\tilde{\phi}}(M'\cap\widetilde{M})=M'\cap\widetilde{M}$ since
\begin{equation*}
y\sigma_{t}^{\tilde{\phi}}(x)
=\sigma_{t}^{\tilde{\phi}}(\sigma_{t^{-1}}^{\tilde{\phi}}(y)x)
=\sigma_{t}^{\tilde{\phi}}(x\sigma_{t^{-1}}^{\tilde{\phi}}(y))
=\sigma_{t}^{\tilde{\phi}}(x)y
\end{equation*}
for any $y\in M$ and $x\in M'\cap\widetilde{M}$.
 Hence, there exists the $\tilde{\phi}$-preserving conditional expectation from $\widetilde{M}$ onto $M'\cap\widetilde{M}$. The restriction of this map gives a conditional expectation from $N'\cap \widetilde{M}$ onto $M'\cap\widetilde{M}$, and this implies injectivity of $M'\cap\widetilde{M}$.\\
(2) Let $M=M_{1}\mathbin{\bar{\otimes}}M_{2}$ be a tensor decomposition of $M$ and assume $M_{2}$ is not of type I. Then, $M_{2}$ has the decomposition $M_{2}=N_{\rm I}\oplus N$, where $N_{\rm I}$ is of type I and $N$ has no type I direct summand. 
By the first part of the proof, $N'\cap 1_{N}\widetilde{M}1_{N}$ is injective.
Since the inclusion 
\begin{equation*}
M_{1}=M_{1}\otimes\mathbb{C}1_{N}\subset M_{1}\mathbin{\bar{\otimes}}1_{N}M_{2}1_{N}=1_{N}M1_{N}\subset 1_{N}(A\mathbin{\bar{\rtimes}}\Gamma)1_{N}
\end{equation*}
has a conditional expectation and $M_{1}$ is contained in $N'\cap 1_{N}(A\mathbin{\bar{\rtimes}}\Gamma)1_{N}$, there is a conditional expectation from $N'\cap 1_{N}(A\mathbin{\bar{\rtimes}}\Gamma)1_{N}$ onto $M_{1}$. This implies injectivity of $M_{1}$.

Thus, we proved that $M_{1}$ is injective if $M_{2}$ is not of type I. By symmetry, it holds that $M_{2}$ is injective if $M_{1}$ is not of type I, and this implies that $M$ is prime.
\end{proof}

In the case where $A$ is non-abelian in the corollary above, a similar statement is still true. Ozawa has already mentioned it for finite von Neumann algebras in $\cite[\rm Remark\ 4.8]{solid 3}$, but his proof had a gap. Here we give a proof of his claim using Theorem $\ref{bi ex thm}$ and our result includes a non-finite case. 

In this setting, the primeness of the crossed product algebra is no longer true since there is a simple counter example; for $A\mathbin{\bar{\rtimes}}\Gamma$ and the AFD $\rm II_1$ factor $R$ (with trivial $\Gamma$-action), $(R\mathbin{\bar{\otimes}}A)\mathbin{\bar{\rtimes}}\Gamma$ is not prime since $R\mathbin{\bar{\otimes}}(A\mathbin{\bar{\rtimes}}\Gamma)
\simeq(R\mathbin{\bar{\otimes}}A)\mathbin{\bar{\rtimes}}\Gamma$.
Hence, we need a $\it correct$ concept in the setting as below. 

\begin{itemize}
\item We say a von Neumann algebra $M$ is $\it semiprime$ if it satisfies the following condition: for any decomposition $M=M_{1}\mathbin{\bar{\otimes}}M_{2}$, one of $M_i$ $(i=1,2)$ is injective.
\end{itemize}
We mention that a finite semiprime factor is prime if it is non McDuff (i.e. it does not have a tensor decomposition with the AFD $\rm II_1$ factor).

\begin{Cor}\label{inj prim}
Let $A$ be a finite injective von Neumann algebra with separeble predual and $\Gamma$ be a countable discrete group acting on $A$. Let $M\subset A\mathbin{\bar{\rtimes}}\Gamma$ be a (possibly non-unital) von Neumann subalgebra with expectation and denote $\widetilde{M}:=1_{M}(A\mathbin{\bar{\rtimes}}\Gamma)1_{M}$. Assume $\Gamma$ is bi-exact.
\begin{itemize}
\item[$(1)$]If $M$ has a finite unital subalgebra with expectation which has no injective direct summand, then the relative commutant $M'\cap\widetilde{M}$ is injective.
\item[$(2)$]
If $M$ is a factor and not of type ${\rm III}_1$, then it is semiprime.
\end{itemize}

\end{Cor}
\begin{proof}

(1) Let $N\subset M$ be a von Neumann subalgebra as in the statement. Since $N$ has no injective direct summand, we have $N\not\preceq_{A\mathbin{\bar{\rtimes}}\Gamma}A$. By Proposition $\ref{cor F}$, there exists an abelian subalgebra $B\subset N$ with the same property as $N$. Applying Theorem $\ref{bi ex thm}$ to $B$, we have injectivity of $B'\cap \widetilde{M}$. Then injectivity of $M'\cap \widetilde{M}$ follows from the same manner as in the proof of Corollary $\ref{ab prim}$(1), since the inclusion $B\subset M\subset\widetilde{M}$ is with expectation.\\
(2) Let $M=M_{1}\mathbin{\bar{\otimes}}M_{2}$ be a tensor decomposition of $M$. We first prove that $M_2$ is injective if $M_1$ has a non-injective semifinite (non-unital) subalgebra $N$ with expectation. 

In this setting, it is easy to find a non-injective $\it finite$ (non-unital) subalgebra $N$ with expectation. Let $z$ be a maximal projection satisfying $zN$ is injective so that $(1-z)N$ has no injective direct summand. Thus we may assume that $N$ is finite and  has no injective direct summand, and so $N'\cap 1_{N}\widetilde{M}1_{N}$ is injective by (1).
Then $M_{2}$ is injective since 
\begin{equation*}
M_{2}\simeq\mathbb{C}1_{N}\mathbin{\bar{\otimes}}M_{2}
\subset N'\cap 1_{N}\widetilde{M}1_{N}
\subset (1_{N}M_{1}1_{N})\mathbin{\bar{\otimes}}M_{2}
\end{equation*}
and $(1_{N}M_{1}1_{N})\mathbin{\bar{\otimes}}M_{2}$ has a conditional expectation onto $\mathbb{C}1_{N}\mathbin{\bar{\otimes}}M_{2}$. 

Now we treat the general case. The case that $M$ is of type $\rm II$ is now trivial, and so we assume that $M$ is of type ${\rm III}_{\lambda}$ $(0\leq\lambda<1)$. For any decomposition $M=M_{1}\mathbin{\bar{\otimes}}M_{2}$ with $M_1$ non-injective, if $M_{1}$ is of type $\rm II$ or type ${\rm III}_{\lambda}$ $(0\leq\lambda<1)$, we have injectivity of $M_2$ since ${\rm III}_{\lambda}$ factor $(0\leq\lambda<1)$ has a discrete decomposition $N\mathbin{\bar{\rtimes}}\mathbb{Z}$, where $N$ is non-injective and of type $\rm II$. Hence we can end the proof if $M$ does not have a decomposition $M_{1}\mathbin{\bar{\otimes}}M_{2}$ with both of $M_i$ type $\rm III_1$. However, this does not happen since $M$ is not of type $\rm III_1$ (e.g. $\cite[\rm Corollary\ 6.8]{CT}$).
\end{proof}

Finally, we give some explicit examples. 

\begin{Exa}\upshape
Let $\partial\mathbb{F}_n$ be the Gromov boundary of $\mathbb{F}_n$ $(2\leq n<\infty)$. Then it is well-known that the natural $\mathbb{F}_n$-action on $\partial\mathbb{F}_n$ is amenable $\cite{Ad}$ and $\partial\mathbb{F}_n$ has a natural Borel measure so that the crossed product $L^{\infty}(\partial\mathbb{F}_n)\mathbin{\bar{\rtimes}}\mathbb{F}_{n}$ is a type III factor (see for example $\cite{RR}$). By amenability of the action, it is an injective type III factor.
\end{Exa}


\begin{Exa}\upshape
For any non-amenable bi-exact group $\Gamma$, the Bernoulli shift on $A:=\bigotimes R$ or $\bigotimes L^\infty(X)$ is an example. But the primeness of this case was already solved by Popa in a general setting $\cite{Po Ber}$.


Here is a similar classical example. Let $\Gamma$ be the wreath product group $(\mathbb{Z}/2\mathbb{Z})\wr\mathbb{F}_{2}$ and $X$ be the standard Borel space $\prod_{g\in\mathbb{F}_{2}}(\mathbb{Z}/2\mathbb{Z},\mu)_{g}$, where $\mu$ is the point measure which has a different value on $\mathbb{Z}/2\mathbb{Z}$. 
Define an affine $\Gamma$-action $\sigma$ on $X$ by $\sigma((x_{g})_{g},s)(y_{g})_{g}:=(y_{s^{-1}g}+x_{g})_{g}$. 
This action was studied by $\rm Puk\grave{a}nsky$ and he proved that the action is free, ergodic, and the crossed product von Neumann algebra $L^{\infty}(X)\mathbin{\bar{\rtimes}}\Gamma$ is of type III $\cite{Pu}$. Moreover it is non-injective since it contains a non-injective subalgebra $L^{\infty}(X)\mathbin{\bar{\rtimes}}\mathbb{F}_{2}$ with expectation.
Thus, $L^{\infty}(X)\mathbin{\bar{\rtimes}}\Gamma$ is a type III prime factor.
\end{Exa}

\begin{Exa}\upshape\label{fin ex}
Let $M_{\alpha}(\Gamma)$ be a crossed product $\rm{II}_{1}$ factor of the form $M_{\alpha}(\Gamma)=R_{\alpha}\mathbin{\bar{\rtimes}}\Gamma$, where $\alpha\in \mathbb{T}$, $R_{\alpha}$ is the finite von Neumann algebra generated by two unitaries $u,v$ satisfying $uv=\alpha vu$, $\Gamma$ is a non-amenable subgroup of ${\rm SL}(2,\mathbb{Z})$. This factor was studied by Nicoara, Popa and Sasyk $\cite{NPS}$, $\cite{Po 3}$. 
By the corollary above, every von Neumann subalgebra of this is semiprime. In particular, every non-McDuff subfactor is  prime.
\end{Exa}

\begin{Exa}\upshape
This example was studied in $\cite{HV}$ by Houdayer and Vaes. Let $\Gamma$ and $\Lambda$ be a countable discrete group and $\pi\colon \Gamma\rightarrow \Lambda$ be a surjective homomorphism. Let $(A,\tau)$ be a finite von Neumann algebra with a distinguished trace and $(Y,\nu)$ be a standard probability space. Assume $\Gamma$ has a free (properly outer) $\tau$-preserving action on $A$ and $\Lambda$ has a free ergodic non-singular action on $(Y,\nu)$.
Consider a diagonal $\Gamma$-action on $A\mathbin{\bar{\otimes}}L^\infty(Y,\nu)$ via $\pi$. Then it is still free and moreover centrally ergodic, if the $(\ker\pi)$-action on $A$ is so. Thus $M:=(A\mathbin{\bar{\otimes}}L^\infty(Y,\nu))\mathbin{\bar{\rtimes}}\Gamma$ is a factor. Denote $N:=L^\infty(Y,\nu)\mathbin{\bar{\rtimes}}\Lambda$ and notice that the unital inclusion $N\subset M$ is with expectation. 

\begin{Lem}
Keep the setting above. Assume the $(\ker\pi)$-action on $A$ is centrally ergodic and $(Y,\nu)$ is non-atomic. Then $M$ and $N$ have the same type and the same flow of weights.
\end{Lem}
\begin{proof}
Here we only prove $M$ is of type $\rm II_1$, $\rm II_\infty$, and $\rm III$ if and only if so is $N$ respectively. The final assertion can be proved in the same way as $\cite[\rm Corollary\ B]{HV}$.

Since $N$ is diffuse and $N\subset M$ is with expectation, $M$ is not of type I. It is easy to see that $M$ is finite if and only if so is $N$. Non-finite cases follows from $\cite[\rm Lem\ V.2.29]{Tak 1}$.
\end{proof}
Now assume that $\Gamma$ is non-amenable and bi-exact and $\Lambda$ is amenable. Then $\ker\pi$ is non-amenable so that $N$ is non-injective. Since $N\subset M$ is with expectation, $M$ is non-injective. Then $M$ is a non-injective semiprime factor if $N$ is not of type ${\rm III}_1$, and is a prime factor if $A$ is abelian. Hence we can construct various types of (semi)prime factors if we can control the initial action on $(Y,\nu)$.

For example, set $\Gamma:=\mathbb{F}_n$, $\Lambda:=\mathbb{Z}$, $\pi$ is a homomorphism via the abelization, and the $\Gamma$-action is the Bernoulli action on $A:=\bigotimes R$ or $\bigotimes L^\infty(X)$. Then $M$ satisfies the assumptions above for any $\mathbb{Z}$-action on any $(Y,\nu)$. Here $\Gamma$ can be chosen as, for example  $\mathbb{Z}^2\rtimes\mathbb{F}_n$ (via $\mathbb{F}_n\subset{\rm SL}(2,\mathbb{Z})$), $\Delta\wr\mathbb{F}_n$ for amenable $\Delta$, or fundamental groups $\pi_1(\Sigma_g)$ $(g\geq2)$.
\end{Exa}

$\bf Acknowledgement$ The author would like to thank Professor Yasuyuki Kawahigashi, who is his
adviser, Masato Mimura, Professor Narutaka Ozawa, and Hiroki Sako for their valuable comments. He was supported by Research Fellow of the Japan Society for the Promotion of Science.



\end{document}